\documentclass[a4paper,oneside,11pt]{article}

\usepackage[a4paper]{geometry}
\usepackage{aeguill}                 
\usepackage{graphicx}
\usepackage{amsmath,amsfonts,amssymb,amsthm}
\usepackage{pstricks,comment} 
\usepackage{epstopdf}
\usepackage{srcltx}
\usepackage[utf8]{inputenc} 
\usepackage[T1]{fontenc} 
 
\usepackage[english]{babel}

\title{On the topological dimension of the Gromov boundaries of some hyperbolic $\text{Out}(F_N)$-graphs}
\author{Mladen Bestvina, Camille Horbez and Richard D. Wade}

\usepackage[all]{xy}
\usepackage{bbm}

\begin{document}

\maketitle
\newtheorem{de}{Definition} [section]
\newtheorem{theo}[de]{Theorem} 
\newtheorem{prop}[de]{Proposition}
\newtheorem{lemma}[de]{Lemma}
\newtheorem{cor}[de]{Corollary}
\newtheorem{propd}[de]{Proposition-Definition}

\theoremstyle{remark}
\newtheorem{rk}[de]{Remark}
\newtheorem{ex}[de]{Example}
\newtheorem{question}[de]{Question}

\normalsize
\newcommand{\note}[1]{\footnote{#1}\marginpar{$\leftarrow$}}
\newcommand{\Ccom}[1]{\Cmod\marginpar{\color{red}\tiny #1 --ch}} 
\newcommand{\Cmod}{$\textcolor{red}{\clubsuit}$}
\newcommand{\mb}[1]{\marginpar{\color{blue}\tiny #1 --mb}}
\newcommand{\rcom}[1]{\rmod\marginpar{\color{purple}\tiny #1 --rw}}
\newcommand{\rmod}{$\textcolor{purple}{\spadesuit}$}
\newcommand{\ccom}[1]{\cnew\marginpar{\color{blue}\tiny #1 --ch}} 
\newcommand{\cnew}{$\textcolor{blue}{\clubsuit}$}

\addtolength\topmargin{-.5in}
\addtolength\textheight{1.in}
\addtolength\oddsidemargin{-.045\textwidth}
\addtolength\textwidth{.09\textwidth}

\newcommand{\imp}{\Rightarrow}
\newcommand{\ra}{\rightarrow}
\newcommand{\m}{^{-1}}
\newcommand{\dunion}{\sqcup}
\newcommand{\eps}{\varepsilon}
\renewcommand{\epsilon}{\varepsilon}
\newcommand{\calc}{\mathcal{C}}
\newcommand{\calf}{\mathcal{F}}
\newcommand{\caly}{\mathcal{Y}}
\newcommand{\calb}{\mathcal{B}}
\newcommand{\calq}{\mathcal{Q}}
\newcommand{\Z}{\mathcal{Z}}
\newcommand{\calh}{\mathcal{H}}
\newcommand{\calo}{\mathcal{O}}
\newcommand{\calu}{\mathcal{U}}
\newcommand{\calv}{\mathcal{V}}
\newcommand{\calw}{\mathcal{W}}
\newcommand{\calt}{\mathcal{T}}
\newcommand{\calx}{\mathcal{X}}
\newcommand{\calz}{\mathcal{Z}}
\newcommand{\cald}{\mathcal{D}}
\newcommand{\cale}{\mathcal{E}}
\newcommand{\cala}{\mathcal{A}}
\newcommand{\calp}{\mathcal{P}}
\newcommand{\calg}{\mathcal{G}}
\newcommand{\bbR}{\mathbb{R}}
\newcommand{\bbZ}{\mathbb{Z}}
\newcommand{\actson}{\curvearrowright}
\newcommand{\es}{\emptyset}
\newcommand{\grp}[1]{\langle #1 \rangle}
\newcommand{\dg}{\dagger}
\newcommand{\stab}{{\rm{Stab}}}
\newcommand{\G}{\Gamma}

\def\R{{\mathbb R}}
\makeatletter
\edef\@tempa#1#2{\def#1{\mathaccent\string"\noexpand\accentclass@#2 }}
\@tempa\rond{017}
\makeatother

\begin{abstract}
We give upper bounds, linear in the rank, to the topological dimensions of the Gromov boundaries of the intersection graph, the free factor graph and the cyclic splitting graph of a finitely generated free group.
\end{abstract}

\setcounter{tocdepth}{2}
\tableofcontents
\section*{Introduction}

The curve graph $\calc(\Sigma)$ of an orientable hyperbolic surface of finite type $\Sigma$ is an essential tool in the study of the mapping class group of $\Sigma$. It was proved to be Gromov hyperbolic by Masur--Minsky \cite{MM99}, and its Gromov boundary was identified with the space of ending laminations on $\Sigma$ by Klarreich \cite{Kla}. A striking application of the curve graph and its geometry at infinity is the recent proof by Bestvina--Bromberg--Fujiwara \cite{BBF} that $\text{Mod}(\Sigma)$ has finite asymptotic dimension, which implies in turn that it satisfies the integral Novikov conjecture.  A crucial ingredient in the proof is the finite asymptotic dimension of the curve graph: this was first proved by Bell--Fujiwara \cite{BF} using Masur--Minsky's \emph{tight geodesics}, and recently recovered (with a bound linear in the genus of $\Sigma$ and the number of punctures) by Bestvina--Bromberg \cite{BB}, via finite capacity dimension of the Gromov boundary. The latter approach builds upon work of Gabai \cite{Gab}, who bounded the topological dimension of $\partial_\infty \mathcal{C}(\Sigma)$ by building covers in terms of train-tracks on the surface; bounding the capacity dimension requires getting more metric control on such covers.

\indent The importance of the curve graph in the study of mapping class groups led people to look for hyperbolic graphs with $\text{Out}(F_N)$-actions. In the present paper, we will be mainly interested in three of them: the \emph{free factor graph} $FF_N$, the \emph{cyclic splitting graph} $FZ_N$ and the \emph{intersection graph} $I_N$. The Gromov boundary of each graph is homeomorphic to a quotient of a subspace of the boundary $\partial CV_N$ of outer space \cite{BR,Ham,Hor,DT}. As $\partial CV_N$ has dimension equal to $3N-5$ \cite{BF94, GL95} and the quotient maps are cell-like, this bounds the cohomological dimension of each of these Gromov boundaries. \emph{A priori} this does not imply finiteness of their topological dimensions. This is the goal of the present paper. 

\theoremstyle{plain}
\newtheorem*{theo:1}{Main Theorem}

\begin{theo:1}\label{main}
Let $N\ge 2$.
\\ The boundary $\partial_\infty I_N$ has topological dimension at most $2N-3$.
\\ The boundary $\partial_\infty FF_N$ has topological dimension at most $2N-2$.
\\ The boundary $\partial_\infty FZ_N$ has topological dimension at most $3N-5$.
\end{theo:1}  

We do not know whether equality holds in any of these cases (see the open questions at the end of the introduction). Following Gabai's work, our proof relies on constructing  a decomposition of each Gromov boundary in terms of a notion of \emph{train-tracks}. We hope that, by getting further control on the covers we construct, this approach may pave the way towards a proof of finite asymptotic dimension of $FF_N$, $FZ_N$ or $I_N$ (via finite capacity dimension of their Gromov boundaries).  
\\
\\
\indent The rest of the introduction is devoted to explaining the strategy of our proof. Although we treat the cases of $\partial_\infty I_N,\partial_\infty FF_N$ and $\partial_\infty FZ_N$ all at once in the paper, we will mainly focus on the free factor graph in this introduction for simplicity, and only say a word about how the proof works for $FZ_N$ (the proof for $I_N$ is similar).

\paragraph*{Bounding the topological dimension.} To establish our main theorem, we will use the following two topological facts  \cite[Lemma 1.5.2 and Proposition 1.5.3]{Eng}. Let $X$ be a separable metric space.
\begin{enumerate}
\item If $X$ can be written as a finite union $X=X_0\cup X_1\cup\dots\cup X_k$ where each $X_i$ is $0$-dimensional, then $\dim(X)\le k$.
\end{enumerate}
Zero-dimensionality of each $X_i$ will be proved by appealing to the following fact.
\begin{enumerate}
\item[2.] If there exists a countable cover of $X_i$ by closed $0$-dimensional subsets, then $\text{dim}(X_i)=0$.
\end{enumerate}

For example, one can recover the fact that $\dim(\mathbb{R}^n)\le n$ from the above two facts, by using the decomposition of $\mathbb{R}^n$ where $X_i$ is the set of points having exactly $i$ rational coordinates. Zero-dimensionality of each $X_i$ can be proved using the second point, by decomposing each $X_i$ into countably many closed subsets $X_i^j$, two points being in the same set precisely when they have the same rational coordinates (one easily checks that these sets $X_i^j$ are $0$-dimensional by finding arbitrarily small boxes around each point of $X^j_i$ with boundary outside of $X_i^j$).

Notice that a decomposition (hereafter, a \emph{stratification}) of $X$ into finitely many subsets $X_i$ as in the first point can be provided by a map $X\to\{0,\dots,k\}$ (hereafter, an \emph{index map}).   Showing that each $X_i$ is zero-dimensional amounts to proving that every point in $X_i$ has clopen neighborhoods (within $X_i$) of arbitrary small diameter (equivalently, every point in $X_i$ has arbitrary small open neighborhoods in $X_i$ with empty boundary). Actually, thanks to the second point above, it is enough to write each stratum $X_i$ as a countable union of closed subsets (called hereafter a \emph{cell decomposition}) and prove $0$-dimensionality of each cell of the decomposition.

\paragraph*{Stratification and cell decomposition of $\partial_\infty FF_N$.} The boundary $\partial_\infty FF_N$ is a separable metric space when equipped with a visual metric. Points in the boundary are represented by $F_N$-trees, and a first reasonable attempt could be to define a stratification of $\partial_\infty FF_N$ by using an index map similar to the one introduced by Gaboriau--Levitt in \cite{GL95}, that roughly counts orbits of branch points and of directions at these points in the trees. Although we make use of these features, our definition of the stratification is slightly different; it is based on a notion of train-tracks.

A train-track $\tau$ (see Definition~\ref{d:train-track}) consists of a triple $(S^\tau,\sim_V,\sim_D)$, where $S^\tau$ is an $F_N$-action on a simplicial tree, and $\sim_V$ (respectively $\sim_D$) is an equivalence relation on the set of vertices (respectively directions) in $S^\tau$.  We say that a tree $T\in\overline{cv}_N$ is \emph{carried} by $\tau$ if there is an $F_N$-equivariant map $f:S^\tau\to T$ (called a \emph{carrying map}) that is in a sense compatible with the train-track structure. The typical situation is when $f$ identifies two vertices in $S^\tau$ if and only if they are $\sim_V$-equivalent, and identifies the germs of two directions based at equivalent vertices in $S^\tau$ if and only if the directions are $\sim_D$-equivalent. For technical reasons, our general definition of a carrying map is slightly weaker; in particular it is a bit more flexible about possible images of vertices in $S^\tau$ at which there are only three equivalence classes of directions. A point $\xi\in\partial_\infty FF_N$ -- which is an equivalence class of trees that all admit alignment-preserving bijections to one another -- is then \emph{carried} by $\tau$ if some (equivalently, any) representative is carried. If $\tau$ carries a point in $\partial_\infty FF_N$ then its underlying tree $S^\tau$ has a free $F_N$-action and determines an open simplex in the interior of outer space .

A train-track $\tau$ as above has an index $i(\tau)$ of at most $2N-2$: this is a combinatorial datum which mainly counts orbits of equivalence classes of vertices and directions in $\tau$. The \emph{index} $i(\xi)$ of a point $\xi\in\partial_\infty FF_N$ is then defined as the maximal index of a train-track that carries $\xi$. A train-track determines a \emph{cell} $P(\tau)$ in $\partial_\infty FF_N$, defined as the set of points $\xi\in\partial_\infty FF_N$ such that $\tau$ carries $\xi$ and $i(\tau)=i(\xi)$. 

We define a stratification of $\partial_\infty FF_N$ by letting $X_i$ be the collection of points of index $i$. Each $X_i$ is covered by the countable collection of the cells $P(\tau)$, where $\tau$ varies over all train-tracks of index $i$. In view of the topological facts recalled above, our main theorem follows for $\partial_\infty FF_N$ from the following points:

\begin{enumerate}
\item (Proposition \ref{closed}) The boundary of $P(\tau)$ in $\partial_\infty FF_N$ is contained in a union of cells of strictly greater index. This implies that $P(\tau)$ is closed in $X_{i(\tau)}$.
\item (Proposition \ref{dim-0}) Each cell $P(\tau)$ has dimension at most $0$.
\end{enumerate}

We say that $P(\tau)$ has dimension `at most' $0$, and not `equal to' $0$, as we do not exclude the possibility that $P(\tau)$ is empty. 

For the first point, we show that if $(\xi_n)_{n\in\mathbb{N}}\in P(\tau)^{\mathbb{N}}$ converges to $\xi\in\partial P(\tau)$, then the carrying maps $f_n $ from $S^\tau$ to representatives of $\xi_n$ converge to a map $f$ from $S^\tau$ to a representative $T$ of $\xi$. However, the limiting map $f$ may no longer be a carrying map: for example, inequivalent vertices in $S^\tau$ may have the same image in the limit, and edges in $S^\tau$ may be collapsed to a point by $f$. We then collapse all edges in $S^\tau$ that map to a point under $f$ and get an induced map $f'$ from the collapse $S'$ to $T$. This determines a new train-track structure $\tau'$ on $S'$. A combinatorial argument enables us to count the number of directions `lost' when passing from $\tau$ to $\tau'$, and show that $i(\tau')>i(\tau)$ unless $T$ is carried by $\tau$. 

Our proof of the second point relies on a \emph{cell decomposition} process similar to Gabai's \cite{Gab}, who used \emph{splitting sequences} of train-tracks on surfaces to get finer and finer covers of the space of ending laminations. In our context, starting from a cell $P(\tau)$, we will construct finer and finer decompositions of $P(\tau)$ into clopen subsets by means of \emph{folding sequences} of train-tracks. Eventually the subsets in the decomposition have small enough diameter. 

More precisely, starting from a train-track $\tau$ one can `resolve' an illegal turn in $\tau$ by folding it. There are several possibilities for the folded track (see the figures in Section~\ref{sec-dim}), so this operation yields a subdivision of $P(\tau)$ into various $P(\tau^{j})$. A first crucial fact we prove is that these $P(\tau^{j})$ are all open in $P(\tau)$.  Now, if $\xi \in P(\tau)$ and $\epsilon >0$, then by folding the train-track $\tau$ for sufficiently long we can reach a train-track $\tau'$ with $diam(P(\tau'))<\epsilon$. Here, hyperbolicity of $FF_N$ is crucial: our folding sequence will determine an (unparameterized) quasi-geodesic going to infinity (towards $\xi$) in $FF_N$. The set $P(\tau_i)$ defined by the train-track obtained at time $i$ of the process is contained in the set of endpoints of geodesic rays in $FF_N$ starting at the simplicial tree associated to $\tau_0$ and passing at a bounded distance from $\tau_i$. From the definition of the visual metric on the boundary, this implies that the diameter of $P(\tau_i)$ converges to $0$ as we move along the folding path. All in all, we find a train-track $\tau'$ such that $P(\tau')$ is an open neighborhood of $\xi$ in $P(\tau)$, and $diam(P(\tau'))< \epsilon$. The boundary of $P(\tau')$ in $P(\tau)$ is empty (it contains points of strictly greater index), so we get $\dim(P(\tau))=0$, as required.

\paragraph*{A word on $\partial_\infty FZ_N$.} Some points in $\partial_\infty FZ_N$ are represented by trees in which some free factor system of $F_N$ is elliptic. This leads us to work with train-tracks that are allowed to have bigger vertex stabilizers, and the index of a train-track now also takes into account the complexity of this elliptic free factor system. \emph{A priori}, our definition of index in this setting only yields a quadratic bound in $N$ to the topological dimension of $\partial_\infty FZ_N$, but we then get the linear bound from cohomological dimension. We also have to deal with the fact that a point $\xi\in\partial_\infty FZ_N$ is an equivalence class of trees that may not admit alignment-preserving bijections between different representatives; this leads us to analysing preferred (mixing) representatives of these classes more closely.

\paragraph*{Organization of the paper.} In Section \ref{sec-background}, we review the definitions of $FF_N$, $I_N$ and $FZ_N$, and the descriptions of their Gromov boundaries. We also prove a few facts concerning mixing representatives of points in $\partial_\infty FZ_N$, which are used to tackle the last difficulty mentioned in the above paragraph. Train-tracks and indices are then defined in Section~\ref{sec-strat}, and the stratifications of the Gromov boundaries are given. We prove in Section~\ref{sec-closed} that each cell $P(\tau)$ is closed in its stratum, by showing that $\partial P(\tau)$ is made of points with strictly higher index. We introduce folding moves in Section \ref{sec-dim}, and use them to prove that each cell $P(\tau)$ has dimension at most $0$. The proof of our main theorem is then completed in Section~\ref{sec-proof}. The paper has an appendix in which we illustrate the motivations behind some technical requirements that appear in our definitions of train-tracks and carrying maps.

\paragraph*{Some open questions.} As mentioned earlier, we hope that the cell decompositions of the boundaries defined in the present paper will provide a tool to tackle the question of finiteness of the asymptotic dimension of the various graphs, following the blueprint of \cite{BB}.

Gabai uses the cell decomposition of the ending lamination space $\mathcal{EL}$ to show that it is highly connected for sufficiently complicated surfaces \cite{Gab} (connectivity had been established earlier for most surfaces by Leininger--Schleimer \cite{LS}). It is unknown whether the boundaries of $I_N$, $FF_N$ or $FZ_N$ are connected.  The question of local connectivity of each boundary (or the related question of local connectivity of the boundary of Culler--Vogtmann's Outer space) is also open to our knowledge. 

We would also like to address the question of finding lower bounds on the dimensions. Gabai shows that the topological dimension of $\mathcal{EL}$ is bounded from below by $3g+p-4$ (where $g$ is the genus of the surface and $p$ is the number of punctures). Since the ending lamination space of a once-punctured surface sits as a subspace of $\partial_\infty FF_{2g}$ and $\partial_\infty FZ_{2g}$, this gives a lower bound on the dimension of $\partial_\infty FF_{2g}$ and $\partial_\infty FZ_{2g}$. Improving the gap between the upper and lower bounds, as well as finding a lower bound for $\partial_\infty I_N$, is an interesting problem.

 The free splitting complex has been notably absent from the above discussion. This is because its boundary is harder to describe, and does not appear as a quotient of a subspace of the boundary of outer space. However, we expect that our methods can be applied here when its boundary is better understood.   One can also apply these techniques to the boundaries of relative versions of the various $\mathrm{Out}(F_N)$-complexes \cite{Hor, GH}. 

\paragraph*{Recent developments.} After the first version of this paper was completed, Guirardel and the first two named authors proved a related result about the asymptotic geometry of $\mathrm{Out}(F_N)$, namely that $\mathrm{Out}(F_N)$ is boundary amenable \cite{BGH}. In particular, we now know that $\mathrm{Out}(F_N)$ satisfies the Novikov conjecture. From the surface viewpoint, Hamenst\"adt has shown that asymptotic dimension of the disk graph of a handlebody of genus at least two is at most quadratic in the genus  \cite{MR3932938}. Also, Bestvina--Bromberg--Fujiwara \cite{BBF2} gave a proof that the mapping class group has finite asymptotic dimension that does not directly appeal to the finite asymptotic dimension of the curve complex (although curve complexes and the Masur--Minsky distance formula still play vital roles). Finiteness of the asymptotic dimension of $\mathrm{Out}(F_N)$ and its related complexes is still a wide open problem.

\paragraph*{Acknowledgments.}  We would like to thank
Patrick Reynolds for conversations we had related to the present
project, and Vera Tonić for pointing out to us the reference
\cite{Eng} relating topological and cohomological dimensions. The
present work was supported by the National Science Foundation under
Grant No. DMS-1440140 while the authors were in residence at the
Mathematical Sciences Research Institute in Berkeley, California,
during the Fall 2016 semester. The first author was supported by the
NSF under Grant No. DMS-1607236.

\section{Hyperbolic $\text{Out}(F_N)$-graphs and their boundaries}\label{sec-background}

We review the definitions of three hyperbolic $\text{Out}(F_N)$-graphs (the free factor graph, the intersection graph and the $\calz$-splitting graph) and the descriptions of their Gromov boundaries. The only novelties are some facts concerning mixing $\calz$-averse trees in the final subsection. 

\subsection{The free factor graph}

The \emph{free factor graph} $FF_N$ is the graph (equipped with the simplicial metric) whose vertices are the conjugacy classes of proper free factors of $F_N$. Two conjugacy classes $[A]$ and $[B]$ of free factors are joined by an edge if they have representatives $A,B$ such that $A\subsetneq B$ or $B\subsetneq A$. Its hyperbolicity was proved in \cite{BF}. When $N=2$ each proper free factor is the conjugacy class of a primitive element. In order to make this graph connected, one adds an edge in $FF_2$ between two conjugacy classes if they have representatives that form a basis of $F_2$. 

To describe its Gromov boundary, we first recall that \emph{unprojectivized outer space} $cv_N$ is the space of all $F_N$-equivariant isometry classes of minimal, free, simplicial, isometric $F_N$-actions on simplicial metric trees. Its closure $\overline{cv}_N$ (for the equivariant Gromov--Hausdorff topology) was identified in  \cite{CL,BF94} with the space of all minimal \emph{very small} $F_N$-trees, i.e. $F_N$-actions on $\mathbb{R}$-trees in which all stabilizers of nondegenerate arcs are cyclic (possibly trivial)  and root-closed, and tripod stabilizers are trivial. There exists a coarsely $\text{Out}(F_N)$-equivariant map $\pi:cv_N\to FF_N$, which sends a tree $T$ to a free factor that is elliptic in a tree $\overline{T}$ obtained from $T$ by collapsing some edges to points. 

An $F_N$-tree $T\in\partial cv_N$ is \emph{arational} if no proper free factor of $F_N$ is elliptic in $T$, or acts with dense orbits on its minimal subtree. Two arational trees $T$ and $T'$ are \emph{equivalent} ($T\sim T'$) if there exists an  $F_N$-equivariant alignment-preserving bijection from $T$ to $T'$. We denote by $\mathcal{AT}_N$ the subspace of $\partial cv_N$ made of arational trees. 
Arational trees were introduced by Reynolds in \cite{Rey}, where he proved that every arational $F_N$-tree is either free, or dual to an arational measured lamination on a once-holed surface.

\begin{theo}(Bestvina--Reynolds \cite{BR}, Hamenstädt \cite{Ham})\label{bdy-ff} There is a homeomorphism $$\partial\pi:\mathcal{AT}_N/{\sim}\to\partial_\infty FF_N,$$ which extends the map $\pi$ continuously to the boundary. By this we mean that for all sequences $(S_n)_{n\in\mathbb{N}}\in (cv_N)^{\mathbb{N}}$ converging to a tree $S_\infty\in \mathcal{AT}_N$ (for the topology on $\overline{cv}_N$), the sequence $(\pi(S_n))_{n\in\mathbb{N}}$ converges to $\partial\pi(S_\infty)$ (for the topology on $FF_N\cup\partial_\infty FF_N$).
\end{theo}

We chose to state our main theorem for the case that $N \geq 2$ as the above theorem (and the rest of the paper) still applies in the case $N=2$. However, as $FF_2$ is a Farey graph and $\partial_\infty FF_2$ is a Cantor set, we see that the result when $N=2$ is not optimal. 

\subsection{The intersection graph}

A conjugacy class $\alpha$ of $F_N$ is \emph{geometric} if it is either part of a free basis of $F_N$, or else corresponds to the boundary curve of a once-holed surface with fundamental group identified with $F_N$. The \emph{intersection graph} $I_N$ (with Mann's definition \cite{Man2}, a variation on Kapovich--Lustig's \cite{KL}) is the bipartite graph whose vertices are the simplicial $F_N$-trees in $\partial cv_N$ together with the set of geometric conjugacy classes of $F_N$. A tree $T$ is joined by an edge to a conjugacy class $\alpha$ whenever $\alpha$ is elliptic in $T$. Its hyperbolicity was proved in \cite{Man2}. The intersection graph is also quasi-isometric to Dowdall--Taylor's co-surface graph \cite[Section 4]{DT}. We denote by $\mathcal{FAT}_N\subseteq\mathcal{AT}_N$ the space of free arational trees in $\overline{cv}_N$. Again, there is a coarsely $\text{Out}(F_N)$-equivariant map $\pi:cv_N\to I_N$, which sends a tree $T\in cv_N$ to a tree $\overline{T}\in\partial cv_N$ obtained by collapsing some of the edges of $T$ to points. 

\begin{theo}(Dowdall--Taylor \cite{DT})
There is a homeomorphism $$\partial\pi:\mathcal{FAT}_N/{\sim}\to\partial_\infty I_N,$$ which extends the map $\pi$ continuously to the boundary.
\end{theo}

Continuity of the extension is understood in the same way as in the statement of Theorem~\ref{bdy-ff} above. When $N=2$, the intersection graph $I_2$ is bounded, and $\partial_\infty I_2$ is empty.

\subsection{The cyclic splitting graph}

A \emph{cyclic splitting} of $F_N$ is a simplicial, minimal $F_N$-tree in which all edge stabilizers are cyclic (possibly trivial). The \emph{$\calz$-splitting graph} $FZ_N$ is the graph whose vertices are the $F_N$-equivariant homeomorphism classes of $\calz$-splittings of $F_N$. Two splittings are joined by an edge if they have a common refinement. Its hyperbolicity was proved by Mann \cite{Man}. 

We recall that two trees $T,T'\in\overline{cv}_N$ are \emph{compatible} if there exists a tree $\widehat{T}\in\overline{cv}_N$ that admits alignment-preserving $F_N$-equivariant maps onto both $T$ and $T'$. A tree $T\in\overline{cv}_N$ is \emph{$\mathcal{Z}$-averse} if it is not compatible with any tree $T'\in\overline{cv}_N$ which is compatible with a $\mathcal{Z}$-splitting of $F_N$. We denote by $\calx_N$ the subspace of $\overline{cv}_N$ made of $\calz$-averse trees. Two $\mathcal{Z}$-averse trees are \emph{equivalent} if they are both compatible with a common tree; although not obvious, this was shown in \cite{Hor} to be an equivalence relation on $\calx_N$, which we denote by $\approx$. 

There is an $\text{Out}(F_N)$-equivariant map $\pi:cv_N\to FZ_N$, given by forgetting the metric.

\begin{theo}(Hamenstädt \cite{Ham}, Horbez \cite{Hor})\label{bdy-fz}
There is a homeomorphism $$\partial\pi:\mathcal{X}_N/{\approx}\to\partial_\infty FZ_N,$$ which extends the map $\pi$ continuously to the boundary.
\end{theo}

Each $\approx$-equivalence class in $\calx_N$ has preferred representatives that are mixing. We recall that a tree $T\in\overline{cv}_N$ is \emph{mixing} if for all segments $I,J\subseteq T$, there exists a finite set $\{g_1,\dots,g_k\}\subseteq F_N$ such that $J$ is contained in the union of finitely many translates $g_iI$. 

\begin{theo}(Horbez \cite{Hor})
Every $\approx$-class in $\calx_N$ contains a mixing tree, and any two mixing trees in the same $\approx$-class admit $F_N$-equivariant alignment-preserving bijections between each other. Any $\calz$-averse tree admits an $F_N$-equivariant alignment-preserving map onto every mixing tree in its $\approx$-class.
\end{theo}

The space of arational trees $\mathcal{AT}_N$ is contained in the space of $\calz$-averse trees $\mathcal{X}_N$, and the equivalence relation $\sim$ we have defined on $\mathcal{AT}_N$ is the restriction to $\mathcal{AT}_N$ of the equivalence relation $\approx$ on $\mathcal{X}_N$ (all arational trees are mixing \cite{Rey}). The inclusion $\mathcal{AT}_N \subseteq\mathcal{X}_N$ then induces a subspace inclusion $\partial_\infty FF_N\subseteq\partial_\infty FZ_N$. 

\subsection{More on mixing $\calz$-averse trees}

In this section, we establish a few more facts concerning mixing $\calz$-averse trees and their possible point stabilizers, building on the work in \cite{Hor}. If we were only concerned with the free factor graph, we would not need these results. The reader may decide to skim through the section and avoid the technicalities in the proofs in a first reading. 

\begin{lemma}\label{factor-action}
Let $T\in\overline{cv}_N$ be mixing and $\calz$-averse, and let $A\subseteq F_N$ be a proper free factor. Then the $A$-action on its minimal subtree $T_A\subseteq T$ is discrete (possibly $T_A$ is reduced to a point).  
\\ In particular, if no proper free factor is elliptic in $T$, then $T$ is arational.
\end{lemma}

\begin{proof}
Assume towards a contradiction that $T_A$ is not simplicial. Since $T$ has trivial arc stabilizers, the Levitt decomposition of $T_A$ \cite{Lev} has trivial arc stabilizers (it may be reduced to a point in the case where $T_A$ has dense orbits). Let $B\subseteq A$ be a free factor of $A$ (hence of $F_N$) which is a vertex group of this decomposition. The $B$-minimal subtree $T_B$ of $T$ (which is also the $B$-minimal subtree of $T_A$) has dense orbits, so by \cite[Lemma~3.10]{Rey1} the family $\{g\overline{T_B}|g\in F_N\}$ is a transverse family in $T$. As $T$ is mixing, it is a transverse covering. In addition, by \cite[Corollary 6.4]{Rey} the stabilizer of $T_B$, and therefore the stabilizer of $\overline{T_B}$, is equal to $B$. But by \cite[Proposition 4.23]{Hor}, the stabilizer of a subtree in a transverse covering of a mixing $\calz$-averse tree cannot be a free factor (in fact, it cannot be elliptic in a $\calz$-splitting of $F_N$), so we get a contradiction.
\end{proof}

A collection $\mathcal{A}$ of subgroups of $F_N$ is a \emph{free factor system} if it coincides with the set of nontrivial point stabilizers in some simplicial $F_N$-tree with trivial arc stabilizers. There is a natural order on the collection of free factor systems, by saying that a free factor system $\cala$ is \emph{contained} in a free factor system $\cala'$ whenever every factor in $\cala$ is contained in one of the factors in $\cala'$. 

\begin{prop}\label{description-stab}
Let $T\in\overline{cv}_N$ be mixing and $\calz$-averse. Then either $T$ is dual to an arational measured lamination on a closed hyperbolic surface with finitely many points removed, or else the collection of point stabilizers in $T$ is a free factor system.
\end{prop}

\begin{proof}
First assume that there does not exist any free splitting of $F_N$ in which all point stabilizers of $T$ are elliptic. Then, with the terminology of \cite[Section 5]{Hor1}, the tree $T$ is of surface type (by the same argument as in \cite[Lemma 5.8]{Hor1}). Since $T$ is $\calz$-averse, the skeleton of the dynamical decomposition of $T$ is reduced to a point, in other words $T$ is dual to an arational measured lamination on a surface. 

Assume now that there exists a free splitting $S$ of $F_N$ in which all point stabilizers of $T$ are elliptic. Let $\mathcal{A}$ be the smallest free factor system such that every point stabilizer of $T$ is contained within some factor in $\mathcal{A}$. By Lemma \ref{factor-action}, each factor $A$ in $\mathcal{A}$ acts discretely on its minimal subtree $T_A$. In addition $T_A$ has trivial arc stabilizers because $T$ has trivial arc stabilizers. So either $A$ is elliptic in $T$, or else $T_A$ is a free splitting of $A$. The second situation cannot occur, as otherwise the point stabilizers in $T_A$ would form a free factor system of $A$, contradicting the minimality of $\cala$. Hence $\mathcal{A}$ coincides with the collection of point stabilizers of $T$. 
\end{proof}

We will also establish a characterization of the mixing representatives in a given equivalence class of $\calz$-averse trees. Before that, we start with the following lemma.

\begin{lemma}\label{lem_mix}
Let $T,T'\in\overline{cv}_N$. Assume that $T$ has dense orbits, and that there exists an $F_N$-equivariant alignment-preserving map $p:T\to T'$.
\\ Then either $p$ is a bijection, or else there exists $g\in F_N$ that is elliptic in $T'$ but not in $T$. In the latter case, some point stabilizer in $T'$ is non-cyclic. 
\end{lemma}

\begin{proof}
Notice that surjectivity of $p$ follows from the minimality of $T'$. The collection of all subtrees of the form $p^{-1}(\{x\})$ with $x\in T'$ is a transverse family in $T$. If $p$ is not a bijection, then one of the subtrees $Y$ in this family is nondegenerate, and \cite[Proposition~7.6]{Rey} implies that its stabilizer $A=\stab(p(Y))$ is nontrivial, and that the $A$-action on $Y$ is not discrete. This implies that the minimal $A$-invariant subtree of $Y$ is not reduced to a point, and $A$ contains an element that acts hyperbolically on $T$, while it fixes a point in $T'$. In addition $A$ is not cyclic, so the last assertion of the lemma holds.
\end{proof}

\begin{prop}\label{mixing}
A $\calz$-averse tree $T$ is mixing if and only if for all $T'\approx T$, every element of $F_N$ that is elliptic in $T'$ is also elliptic in $T$.
\\ If $T$ is dual to an arational lamination on a surface, then all trees in the $\approx$-class of $T$ are mixing.
\end{prop}

\begin{proof}
If $T$ is mixing, then all trees $T'$ in its $\approx$-class admit an $F_N$-equivariant alignment-preserving map onto $T$, so every element elliptic in $T'$ is also elliptic in $T$. If $T$ is not mixing, then it admits an $F_N$-equivariant alignment-preserving map $p$ onto a mixing tree $\overline{T}$ in the same $\approx$-class, and $p$ is not a bijection. By Lemma~\ref{lem_mix}, there exists an element of $F_N$ that is elliptic in $\overline{T}$ but not in $T$. The second assertion in Proposition~\ref{mixing} follows from the last assertion of Lemma~\ref{lem_mix} because if $T$ is dual to an arational lamination on a surface, then $T$ is mixing and all point stabilizers in $T$ are cyclic.
\end{proof}

\begin{cor}\label{cor-mixing}
Let $T$ be a $\calz$-averse tree, and let $\cala$ be a maximal free factor system elliptic in $T$. Let $\overline{T}$ be a mixing representative of the $\approx$-class of $T$. 
\\ Then $\cala$ is a maximal elliptic free factor system in $\overline{T}$ if and only if $T$ is mixing.
\end{cor}

\begin{proof}
If $T$ is mixing, then $T$ and $\overline{T}$ have the same point stabilizers, so the conclusion is obvious. If $T$ is not mixing, then Proposition~\ref{mixing} implies that $\overline{T}$ is not dual to a lamination on a surface, and that there exists an element $g\in F_N$ not contained in $\cala$ that is elliptic in $T$. Proposition \ref{description-stab} shows that the collection of elliptic subgroups in $\overline{T}$ is a free factor system of $F_N$, and this free factor system strictly contains $\cala$.
\end{proof}

\section{Train-tracks, indices and stratifications}\label{sec-strat}

\subsection{Train-tracks and carried trees}

\begin{de}[\textbf{\emph{Train-tracks}}] \label{d:train-track}
A \emph{train-track} $\tau$ is the data of 
\begin{itemize}
\item a minimal, simplicial $F_N$-tree $S^\tau$ with trivial edge stabilizers,
\item an equivalence relation $\sim^\tau_V$ on the set $V(S^\tau)$ of vertices of $S^\tau$, such that if $v\sim^\tau_V v'$, then $gv\sim^\tau_V gv'$ for all $g\in F_N$, and such that no two adjacent vertices are equivalent,
\item for each $\sim^\tau_V$-class $X$, an equivalence relation $\sim^\tau_{D,X}$ on the set $\cald(X)$ of directions at the vertices in $X$, such that if two directions $d,d'\in\cald(X)$ are equivalent, then for all $g\in F_N$, the directions $gd,gd'\in\cald(gX)$ are also equivalent.
\end{itemize}
Equivalence classes of directions at $X$ are called \emph{gates} at $X$.
\\ We denote by $\mathcal{A}(\tau)$ the free factor system made of all point stabilizers in $S^{\tau}$. \end{de}

\begin{rk}
Including an equivalence relation on the vertex set of $S^\tau$ in the definition may look surprising to the reader, as this is not standard in train-track theory for free groups. Roughly speaking, the equivalence classes of vertices correspond to branch points in trees carried by a track; this is explained in more detail in the appendix of the paper. 
\end{rk}
 
We will usually also impose that the train-tracks we work with satisfy some additional assumptions. Let $\tau$ be a train-track. A pair $(d,d')$ of directions based at a common vertex of $S^\tau$ is called a \emph{turn}, and is said to be \emph{legal} if $d$ and $d'$ are inequivalent. A subtree $A \subset S^\tau$ \emph{crosses the turn} $(d,d')$ if the intersection of $A$ with both directions $d$ and $d'$ is non-empty. We say that $A$ is \emph{legal} in $\tau$ if each turn crossed by $A$ is legal in $\tau$.

\begin{de}[\textbf{\emph{Admissible train-tracks}}]\label{adm} A train-track $\tau$ is \emph{admissible} if for every vertex $v\in S^\tau$, there exist three pairwise inequivalent directions $d_1,d_2,d_3$ at $v$ such that for all $i\in\{1,2,3\}$, there exists an element $g_i\in F_N$ acting hyperbolically on $S^\tau$ whose axis in $S^\tau$ is legal and crosses the turn $(d_i,d_{i+1})$ (mod $3$). 
\end{de}

In particular, there are at least three gates at every equivalence class of vertices in admissible train-tracks. We will call a triple $(g_1,g_2,g_3)$ as in Definition~\ref{adm} a \emph{tripod of legal elements} at $v$, and we say that their axes form a \emph{tripod of legal axes} at $v$.
\\
\\
\indent Edges in simplicial trees are given an affine structure, which enables us to consider maps from simplicial trees to $\mathbb{R}$-trees that are linear on edges. If $f:S\to T$ is linear then there is a unique metric on $S$ for which $f$ is isometric when restricted to every edge. We say that this is the \emph{metric on $S$ determined by the linear map} $f$.

\indent Train-tracks naturally arise from morphisms between trees. Suppose that $f:S \to T$ is an $F_N$-equivariant map from a simplicial $F_N$-tree $S$ to an $F_N$-tree $T$ which is linear on edges and does not collapse any edge of $S$ to a point. There is an induced equivalence relation $\sim_V$ on the set of vertices of $S$ given by saying that two vertices are equivalent if they have the same $f$-image in $T$.  Let $w$ be a point in $T$ and let $X:=f^{-1}(w)\cap V(S)$ be its associated equivalence class in $V(S)$. As soon as $X\neq\emptyset$, equivariance of $f$ implies that its (setwise) stabilizer $\stab(X)$ in $S$ is equal to the stabilizer of the point $w$ in $T$. Let $\mathcal{D}(X)$ be the set of directions in $S$ based at points in $X$. Since $f$ does not collapse any edge, it induces an equivalence relation $\sim_{D,X}$ on $\mathcal{D}(X)$, where two directions in $\mathcal{D}(X)$ are equivalent if they have germs that map into the same direction at $w$ in $T$. The set of $\stab(X)$-orbits of equivalence classes in $\mathcal{D}(X)$ then maps injectively into the set of $\stab(w)$-orbits of directions at $w$ in $T$. 
We call this collection of equivalence classes $\tau_f$ the \emph{train-track structure on $S$ induced by $f$}. This will be the typical example where a tree $T$ is carried by the train-track $\tau_f$. Our general definition of carrying is slightly more technical, as it involves a bit more flexibility with respect to the definition of $f$ at exceptional classes of vertices, defined as follows. 

\begin{de}[\textbf{\emph{Exceptional classes of vertices}}]
Let $\tau$ be an admissible train-track. An equivalence class $X$ of vertices of $S^\tau$ is called \emph{exceptional} if there are exactly $3$ gates at $X$. 
\end{de}

\begin{de}[\textbf{\emph{Specialization}}]\label{de-specialization} Let $\tau$ and $\tau'$ be two train-tracks, and let $v_0\in S^\tau$ be such that $[v_0]$ is an exceptional equivalence class. We say that $\tau'$ is a \emph{specialization} of $\tau$ at $[v_0]$ if 
\begin{itemize}
\item $S^{\tau'}=S^\tau$,
\item there exists a vertex $v_1\in S^\tau$, not in the same orbit as $v_0$, such that $\sim_V^{\tau'}$ is the coarsest $F_N$-invariant equivalence relation finer than $\sim_V^\tau$ and such that $v_1\sim_V^{\tau'} v_0$,
\item if $d,d'$ are two directions not based at any vertex in the orbit of $[v_0]$, then $d\sim^{\tau'} d'$ if and only if $d\sim^{\tau}d'$, 
\item every direction at a vertex in $[v_0]$ is equivalent to some direction at a vertex which is $\tau$-equivalent to $v_1$. 
\end{itemize}
\end{de}

\begin{de}[\textbf{\emph{Carrying}}]\label{de-carry}
We say that an $F_N$-tree $T$ is \emph{carried} by a train-track $\tau$ if there is an $F_N$-equivariant map $f:S^\tau \to T$, which is linear on edges and does not collapse any edge to a point, such that the train-track structure $\tau_f$ induced by $f$ is obtained from $\tau$ by a finite (possibly trivial) sequence of specializations.

\noindent In this situation, we call the map $f$ a \emph{carrying map} (with respect to $\tau$). 
\end{de}

\begin{rk}
The motivation for introducing specializations and allowing them in the definition of carrying is again explained in the appendix; specializations will appear naturally later in the paper when we start performing folds on tracks.
\end{rk}

\begin{lemma}\label{countable}
The collection of all train-tracks $\tau$ for which there exists a tree $T\in\overline{cv}_N$ carried by $\tau$ is countable.
\end{lemma}

\begin{proof}
There are countably many minimal, simplicial $F_N$-trees with trivial edge stabilizers. In addition, if $\tau$ carries a tree $T\in\overline{cv}_N$, then the stabilizer of every equivalence class $X$ of vertices in $\tau$ is a finitely generated subgroup of $F_N$ (because every point stabilizer in a very small $F_N$-tree is finitely generated by \cite{GL95}). The set $X/\text{Stab}(X)$ is finite  because there are finitely many $F_N$-orbits of vertices in $S^\tau$, and if $v,gv$ are two vertices in $X$, then $g\in\text{Stab}(X)$. The stabilizer of every gate $[d]$ is at most cyclic, and again $[d]/\stab ([d])$ is finite. The equivalence relation on vertices is recovered by taking a finite set $A_1, \ldots, A_k$ of representatives of $F_N$-orbits of equivalence classes, and in each $A_i$ taking a finite set $V_i$ of vertices such that $\stab(A_i).V_i=A_i$. Each equivalence class is of the form $g\stab(A_i).V_i$, for some $g \in F_N$.  The equivalence relation on edges is recovered by taking a finite set $B_1, \ldots, B_l$ of representatives of $F_N$-orbits of gates in $\tau$, and for each $B_i$ taking a finite set $E_i$ of representatives of $F_N$-orbits of oriented edges determining the directions in $B_i$. Then each gate in $\tau$ is of the form $g\stab(B_i).E_i$ for some $g \in F_N$. Hence $\tau$ is determined by the simplicial tree $S^\tau$ and the finite family $(\{\stab{(A_i)}\}, \{ V_i \} , \{\stab(B_i)\}, \{E_i\})$, which gives a countable number of possible train-tracks. 
\end{proof}

We will also need the following observation.

\begin{lemma}\label{recurrent}
Let $\tau$ be an admissible train-track, and let $[a,b]\subseteq S^\tau$ be a legal segment. Then there exists an element $g\in F_N$ acting hyperbolically on $S^\tau$, whose axis in $S^\tau$ is legal and contains $[a,b]$. 
\end{lemma}

\begin{proof}
Without loss of generality, we can assume that $a$ and $b$ are vertices of $S^\tau$. Since $\tau$ is admissible, there exist elements $g,h\in F_N$ acting hyperbolically in $S^\tau$, whose axes $A_g,A_h$ are legal and pass through $a,b$ respectively, only meet $[a,b]$ at one extremity, and such that the subtree $Y:=A_g\cup [a,b]\cup A_h$ is legal. Standard properties of group actions on trees \cite{CM} imply that all turns in the axis of $gh$ are contained in translates of $Y$. So the axis of $gh$ is legal, and it contains $[a,b]$.
\end{proof}

The following two lemmas state that every tree $T\in\overline{cv}_N$  with trivial arc stabilizers is carried by an admissible train-track, and in addition the carrying map $f:S^\tau\to T$ is completely determined by the train-track structure. Given a free factor system $\cala$, we say that a tree $T\in\overline{cv}_N$ is an \emph{$\cala$-tree} if all subgroups in $\cala$ are elliptic in $T$.

\begin{lemma}\label{carrier-exist}
Let $\cala$ be a free factor system, and let $T\in\overline{cv}_N$ be an $\cala$-tree with trivial arc stabilizers.  Assume that $\cala$ is a maximal free factor system such that $T$ is an $\cala$-tree. Then there exists an admissible train-track $\tau$ such that $T$ is carried by $\tau$ and $\cala(\tau)=\cala$.
\end{lemma}

\begin{proof}
We let $\cala/F_N:=\{[A_1],\dots,[A_k]\}$. If $\cala\neq\emptyset$, then (up to replacing the subgroups $A_i$ by appropriate conjugates) we can choose for $S$ the universal cover of the graph of groups depicted in Figure \ref{fig-Grushko}. There is a unique $F_N$-equivariant map $f:S\to T$ that is linear on edges: indeed, every vertex $v$ of $S$ has nontrivial stabilizer $G_v$, and must be sent by $f$ to the unique point in $T$ fixed by $G_v$.  The elements $a_i$ and the subgroups $A_2,\dots,A_k$ cannot fix the same point as $A_1$ in $T$ by maximality of $\cala$, so $f$ does not collapse any edge, and therefore $\tau_f$ is well-defined. Admissibility of $\tau_f$ can be checked by taking advantage of the fact that all vertex groups are infinite: one can construct the required legal elements by taking appropriate products of elliptic elements in $T$. 

We now assume that $\cala=\emptyset$. Let $\{a_1,\dots,a_N\}$ be a free basis of $F_N$ such that $A=\langle a_1,a_2 \rangle$ acts freely with discrete orbits on $T$, and no $a_i$ is elliptic in $T$. Such a basis exists as the reducing factors of $T$ (if there are any) form a bounded subset of $FF_N$, see \cite[Corollary 5.3]{BR}. Let $T_A$ be the minimal $A$-invariant subtree for the $A$-action on $T$. Let $\G_A=T_A/A$. Then $\G_A$ is a finite graph in $cv_2$ with one or two vertices. We extend $\G_A$ to a marked $F_N$--graph $\G$ with the same number of vertices, by attaching loops labelled by the elements $a_3,\dots,a_N$ at a fixed vertex of $\G_A$. Let $S$ be the universal cover of $\G$. The isometric embedding from $T_A$ into $T$ extends to a map $f:S \to T$, and $f$ does not collapse any edges because no element  $a_i$ is elliptic in $T$. The induced train-track $\tau_f$ is admissible as each vertex of $T_A$ has valence at least 3 and $f$ is an $F_N$-equivariant isometric embedding when restricted to $T_A$. 
\end{proof}

\begin{figure}
\begin{center}
\input{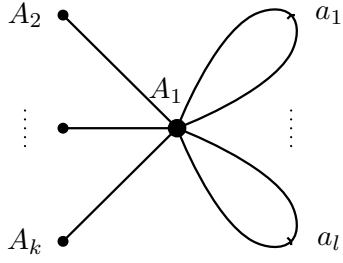}
\caption{The tree $S$ in the proof of Lemma \ref{carrier-exist}, in the case where $\cala\neq\emptyset$.}
\label{fig-Grushko}
\end{center}
\end{figure}

\begin{lemma}\label{carrier-unique}
Let $T\in\overline{cv}_N$, and let $\tau$ be an admissible train-track such that $T$ is carried by $\tau$. Then there exists a unique carrying map $f:S^{\tau}\to T$ for the train-track structure.
\\ Furthermore, the carrying map $f$ varies continuously on the set of trees that are carried by $\tau$ (in the equivariant Gromov--Hausdorff topology).
\end{lemma}

\begin{proof}
To prove uniqueness, it is enough to show that the $f$-image of any vertex $v\in S^\tau$ is completely determined by the train-track structure. Let $(g_1,g_2,g_3)$ be a tripod of legal elements at $v$, and let $e_1,e_2,e_3$ be the three edges at $v$ taken by their axes. 
Then $f$ is isometric when restricted to $e_1\cup e_2\cup e_3$, and also when restricted to the axes of $g_1,g_2,g_3$. This implies that the intersection of the axes of $g_1,g_2,g_3$ in $T$ is reduced to a point, and $f$ must send $v$ to that point. Continuity also follows from the above argument.
\end{proof}

In order to define what it means for an equivalence class of $\calz$-averse trees to be carried by a train-track, we will need the following lemma.

\begin{lemma}\label{carry-collapse-1}
Let $T,T'\in\overline{cv}_N$, and let $\tau$ be a train-track.
\\ If there is an $F_N$-equivariant alignment-preserving bijection from $T$ to $T'$, then $T$ is carried by $\tau$ if and only if $T'$ is carried by $\tau$.
\end{lemma}

\begin{proof}
Let $f:S^\tau\to T$ be a carrying map, and let $\theta:T\to T'$ be an $F_N$-equivariant alignment-preserving bijection. Let $f':S^\tau\to T'$ be the unique linear map that coincides with $\theta\circ f$ on the vertices of $S^\tau$. We claim that $f'$ is a carrying map. Indeed, since $\theta$ is a bijection, two vertices in $S^\tau$ have the same $f'$-image if and only if they have the same $f$-image. In addition, since $\theta$ preserves alignment, the germs of two directions in $S^\tau$ are identified by $f'$ if and only if they are identified by $f$. 
\end{proof}

We recall that $\calx_N\subseteq\overline{cv}_N$ denotes the subspace made of $\calz$-averse trees.

\begin{de}[\textbf{\emph{Carrying equivalence classes of $\calz$-averse trees}}] 
An equivalence class $\xi\in\calx_N/{\approx}$ is \emph{carried} by a train-track $\tau$ if some (equivalently, any) mixing representative of $\xi$ is carried by $\tau$. 
\end{de}

\subsection{Indices and stratifications} 

We now define stratifications of the Gromov boundaries of $I_N$, $FF_N$ and $FZ_N$ by means of an index function taking finitely many values. We will define both the index of a train-track and the index of a tree $T\in\overline{cv}_N$, and explain how the two are related, before defining the index of a boundary point (which is an equivalence class of trees). 

\subsubsection{Geometric index of a tree in $\overline{cv}_N$}

The following definition is reminiscent of Gaboriau--Levitt's index for trees in $\overline{cv}_N$ \cite{GL95}, with a slight difference in the constants we use. 

\begin{de}[\textbf{\emph{Geometric index of a tree $T\in\overline{cv}_N$}}]
Let $T\in\overline{cv}_N$ be a tree with trivial arc stabilizers. The \emph{geometric index} of $T$ is defined as $$i_{geom}(T):=\sum_{v\in V(T)/F_N}(\alpha_v+3~\mathrm{rank}(\text{Stab}(v))-3),$$ where $V(T)$ denotes the set of branch points in $T$, and $\alpha_v$ denotes the number of $F_N$-orbits of directions at $v$.
\end{de}

We warn the reader that the Gaboriau--Levitt index (denoted $i_{GL}$, below) is often also called the geometric index of a tree, and the definition used in this paper is somewhat nonstandard. The two are related in \eqref{eq1}, below (see \cite{MR2944977} for a discussion of various notions of index for trees and automorphisms). 

\begin{lemma}\label{geom-bound}
For all $T\in\overline{cv}_N$ with trivial arc stabilizers, we have $i_{geom}(T)\le 3N-3$. 
\\ If $T$ is arational, then $i_{geom}(T)\le 2N-2$.
\\ If $T$ is free and arational, then $i_{geom}(T)\le 2N-3$.
\end{lemma}

\begin{proof}
Let $$i_{GL}(T):=\sum_{v\in V(T)/F_N}(\alpha_v+2~\text{rank}(\stab (v))-2).$$ We have $i_{geom}(T)\le \frac32i_{GL}(T)$ for all $T\in\overline{cv}_N$, and Gaboriau--Levitt proved in \cite[Theorem III.2]{GL95} that $i_{GL}(T)\le 2N-2$, so $i_{geom}(T)\le 3N-3$. In addition we have 
\begin{equation}\label{eq1}
i_{geom}(T)=i_{GL}(T)+\sum_{v\in V(T)/F_N}(\text{rank}(\text{Stab}(v))-1).
\end{equation}
\noindent Assume now that $T$ is arational. If the $F_N$-action on $T$ is free, since there is at least one orbit of branch points in $T$, we get from Equation~\eqref{eq1} that $i_{geom}(T)<i_{GL}(T)\le 2N-2$. Otherwise, by \cite{Rey}, the $F_N$-action on $T$ is dual to an arational measured lamination on a once-holed surface, so all point stabilizers in $T$ are either trivial or cyclic, and we get $i_{geom}(T)\le i_{GL}(T)\le 2N-2$. 
\end{proof}

\subsubsection{Index of a train-track and carrying index of a tree}

The \emph{height} $h(\cala)$ of a free factor system $\cala$ is defined as the maximal length $k$ of a proper chain of free factor systems $\emptyset=\cala_0\subsetneq\cala_1\subsetneq\dots\subsetneq\cala_k=\cala$. We recall that the free factor system $\cala(\tau)$ associated to a train-track $\tau$ is the free factor system consisting of the vertex stabilizers in the associated tree $S^\tau$.  

\begin{de}[\textbf{\emph{Index of a train-track}}]
The \emph{geometric index} of a train-track $\tau$ is defined as $$i_{geom}(\tau):=\sum_{i=1}^{k}(\alpha_i+3r_i-3),$$ where the sum is taken over a finite set $\{X_1,\dots, X_k\}$ of representatives of the $F_N$-orbits of equivalence classes in $V(S^\tau)$, $\alpha_i$ denotes the number of $\stab(X_i)$-orbits of gates at the vertices in $X_i$, and $r_i$ denotes the rank of  $\stab(X_i)$.
\\ The \emph{index} of $\tau$ is defined to be the pair $i(\tau):=(h(\cala(\tau)),i_{geom}(\tau))$. 
\end{de}

Indices of train-tracks will be ordered lexicographically. 

\begin{rk}
 If $\tau$ is admissible, then every equivalence class $X$ of vertices in $\tau$ has a nonnegative contribution to the geometric index of $\tau$. The contribution of a class $X$ to the geometric index of $\tau$ is zero if and only if $\stab(X)$ is trivial, and there are exactly three gates at $X$; if $\tau$ carries a tree with trivial arc stabilizers then this is precisely when $X$ is exceptional. 
\end{rk}

\begin{lemma}\label{geom-index}
Let $T\in\overline{cv}_N$ be a tree with trivial arc stabilizers, and let $\tau$ be an admissible train-track that carries $T$. Then $i_{geom}(\tau)\le i_{geom}(T)$. 
\end{lemma}

\begin{proof}
Let $f:S^\tau\to T$ be the unique carrying map. Since $\tau$ is admissible, vertices in $S^\tau$ are mapped to branch points in $T$. As $f$ is a carrying map, distinct equivalence classes of nonexceptional vertices in $S^\tau$ are mapped to distinct branch points in $T$. If $X$ is an equivalence class in $S^\tau$ mapping to a branch point $v \in T$ then $\stab(X)=\stab(v)$. Furthermore, one checks that two gates based at $X$ in distinct $\stab(X)$-orbits are mapped under $f$ to directions at $v$ in distinct $\stab(v)$-orbits. Since exceptional vertices do not contribute to the geometric index, it follows that $i_{geom}(\tau) \leq i_{geom}(T)$.  
\end{proof}

Notice that the inequality from Lemma~\ref{geom-index} might be strict if some branch direction in the tree $T$ is not ``visible'' in the track. In the following definition, instead of directly counting branch points and branch directions in $T$, we will count the maximal number of such directions that are visible from a train-track.

\begin{de}[\textbf{\emph{Carrying index of a tree $T\in\overline{cv}_N$}}]
Let $T\in\overline{cv}_N$ be a tree with trivial arc stabilizers. We define the \emph{carrying index} of $T$, denoted by $i(T)$, as the maximal index of an admissible train-track $\tau$ that carries $T$. 
\\ An \emph{ideal carrier} for $T$ is an admissible train-track $\tau$ that carries $T$ such that $i(T)=i(\tau)$.
\end{de}

\begin{rk}\label{rk-index}
In view of Lemma \ref{carrier-exist}, if $\tau$ is an ideal carrier of a tree $T$, then $\cala(\tau)$ is a maximal free factor system that is elliptic in $T$.
\end{rk}

\begin{lemma}
The  carrying index of a tree $T\in\overline{cv}_N$ with trivial arc stabilizers can only take boundedly many values  (with a bound only depending on $N$).
\\ The  carrying index of an arational tree is comprised between $(0,0)$ and $(0,2N-2)$.
\\ The  carrying index of a free arational tree is comprised between $(0,0)$ and $(0,2N-3)$. 
\end{lemma}

\begin{proof}
The first assertion is a consequence of Proposition \ref{geom-bound}, together with the fact that there is a bound (only depending on $N$) on the height of a free factor system of $F_N$. The other assertions follow from Proposition~\ref{geom-bound} because no nontrivial free factor is elliptic in an arational tree, so the height $h(\cala(\tau))$ for a train-track $\tau$ that carries an arational tree is $0$.
\end{proof}

\subsubsection{Index of a boundary point and stratifications}\label{sec-st}

We now define the index of an equivalence class $\xi$ of $\calz$-averse trees. We recall that by definition a train-track $\tau$ carries $\xi$ if and only if $\tau$ carries the mixing representatives of $\xi$.

\begin{de}[\textbf{\emph{Index of a boundary point}}]
The \emph{index} $i(\xi)$ of an equivalence class $\xi\in\calx_N/{\approx}$ is defined as the maximal index of an admissible train-track that carries $\xi$ (equivalently $i(\xi)$ is the carrying index of the mixing representatives of $\xi$).
\\ Given an admissible train-track $\tau$, we define the \emph{cell} $P(\tau)$ as the subspace of $\calx_N/{\approx}$ made of all classes $\xi$ that are carried by $\tau$ and such that $i(\xi)=i(\tau)$.
\\ For all $i$, we define the \emph{stratum} $X_i\subseteq\mathcal{X}_N/{\approx}$ as the set of all points $\xi\in\calx_N/{\approx}$ such that $i(\xi)=i$.  
\end{de}

The Gromov boundary $\partial_\infty FZ_N$ can be written as the union of all strata $X_i$, where $i$ varies over the finite set of all possible indices for mixing $\calz$-averse trees. 

If a tree $T\in\overline{cv}_N$ is mixing and $\calz$-averse, and if no proper free factor is elliptic in $T$, then Lemma \ref{factor-action} says that $T$ is arational. Therefore, the boundary $\partial_\infty FF_N$ coincides with the subspace of $\calx_N/{\approx}$ which is the union of all strata $X_i$ with $i$ comprised between $(0,0)$ and $(0,2N-2)$. 

Finally, the boundary $\partial_\infty I_N$ can be written as the union $X'_0\cup\dots\cup X'_{2N-3}$, where $X'_i$ is the subspace of $X_{(0,i)}$ made of equivalence classes of free actions. 

In view of the topological facts recalled in the introduction, we are left showing that each cell $P(\tau)$ is closed in its stratum $X_{i(\tau)}$, and that $\dim(P(\tau))\le 0$. This will be the contents of Sections \ref{sec-closed} and \ref{sec-dim}, respectively. We give a complete overview of the proof of our main theorem in Section~\ref{sec-proof}.

\section{Closedness of $P(\tau)$ in its stratum}\label{sec-closed}

In general, the property of a tree being carried by a train-track is not a closed condition. However, cells determined by train-tracks are closed in their own strata. This is the goal of the present section. As each boundary we study is metrizable, throughout the paper we will use sequential arguments to work with the topology where this is appropriate. 

\begin{prop}\label{closed}
Let $\tau$ be an admissible train-track. Then all points in $\partial P(\tau) = \overline{P(\tau)} \setminus P(\tau) \subseteq\calx_N/{\approx}$ have index strictly greater than $i(\tau)$. 
\\ In particular $P(\tau)$ is closed in $X_{i(\tau)}$.
\end{prop}

Our proof of Proposition~\ref{closed} is based on Lemma~\ref{lem_stab_up} and Proposition~\ref{prop_index_up} below.

\begin{lemma}\label{lem_stab_up}
Let $\tau$ be an admissible train-track, and let $(T_n)_{n\in\mathbb{N}}$ be a sequence of mixing $\calz$-averse trees carried by $\tau$. Assume that the trees $T_n$ converge to a $\calz$-averse tree $T$, and denote by $\xi$ the $\approx$-class of $T$. 
\\ Then either $T$ is mixing and $\cala(\tau)$ is a maximal free factor system that is elliptic in $T$, or else $i(\xi)>i(\tau)$.
\end{lemma}

\begin{proof}
The free factor system $\cala(\tau)$ is elliptic in all trees $T_n$, and therefore it is also elliptic in $T$. If $\cala(\tau)$ is not a maximal free factor system elliptic in $T$, then as $T$ collapses to any mixing representative $\overline{T} \in \xi$ the free factor system $\cala(\tau)$ is not maximally elliptic in $\overline{T}$. Remark~\ref{rk-index} then implies that $i(\xi)>i(\tau)$. If $T$ is not mixing (but possibly $\cala(\tau)$ is maximally elliptic in $T$), then Corollary~\ref{cor-mixing} implies that again $\cala(\tau)$ is not maximally elliptic in any mixing representative of $\xi$. Hence $i(\xi)>i(\tau)$ in this case, also. 
\end{proof}

To complete the proof of Proposition~\ref{closed}, we are thus left understanding the case where the limiting tree $T$ is mixing, and $\cala(\tau)$ is a maximal elliptic free factor system in $T$: this will be done in Proposition~\ref{prop_index_up} below. The idea is that the carrying maps $f_n:S^\tau\to T_n$ will always converge to an $F_N$-equivariant map $f:S^\tau\to T$; in general the limiting map $f$ can fail to be a carrying map (it may even collapse some edges to points), and in this case we will prove in Proposition~\ref{prop_index_up} that there is a jump in index. We start by proving the existence of the limiting map $f$.

\begin{lemma}
Let $\tau$ be an admissible train-track, and let $(T_n)_{n\in\mathbb{N}}$ be a sequence of trees in $\overline{cv}_N$ with trivial arc stabilizers, converging to a tree $T\in\overline{cv}_N$. Assume that all trees $T_n$ are carried by $\tau$, and for all $n\in\mathbb{N}$, let $f_n:S^\tau\to T_n$ be the carrying map. 
\\ Then the maps $f_n$ converge (in the equivariant Gromov--Hausdorff topology) to a map $f:S^\tau\to T$.
\end{lemma}

\begin{proof}
Let $v$ be a vertex in $S^\tau$. Since $\tau$ is admissible, there exist three inequivalent edges $e_1,e_2,e_3$ that form a legal tripod at $v$, and two elements $g,h\in F_N$ that both act hyperbolically in $S^\tau$, whose axes are legal, and such that the axis of $g$ (resp. $h$) crosses the turn $(e_1,e_3)$ (resp. $(e_2,e_3)$). Then the axes of $g$ and $h$ in $S^\tau$ intersect in a compact non-degenerate segment with initial point $v$, and up to replacing $g$ and $h$ by their inverses, we can assume that $g$ and $h$ both translate along this segment in the direction going out of $v$. Since all maps $f_n$ are carrying maps, the elements $g$ and $h$ are hyperbolic in all trees $T_n$, and their axes in $T_n$ intersect in a compact nondegenerate segment, on which they translate in the same direction. In the limiting tree $T$ the intersection $A^T_g \cap A^T_h$ of the axes (or fixed sets) of $g$ and $h$ is still a compact segment, and if it is non-degenerate the elements $g$ and $h$ are hyperbolic in $T$ and still translate in the same direction along the intersection. We then define $f(v)$ to be the initial point of $A^T_g \cap A^T_h$. We repeat this process over each orbit of vertices to obtain an $F_N$-equivariant map from the vertices of $S^\tau$ to $T$, and extend the map linearly over edges. Distances between intersections of axes (as well as their initial points) are determined by the Gromov--Hausdorff topology \cite{Paulin}, so it follows that for any two vertices $v$ and $v'$, the distance $d_{T_n}(f_n(v),f_n(v'))$ converges to $d_{T}(f(v),f(v'))$ as $n$ goes to infinity. This implies that the sequence of maps $(f_n)$ converges to $f$. 
\end{proof}

\begin{prop}\label{prop_index_up}
Let $\tau$ be an admissible train-track, and let $(T_n)_{n\in\mathbb{N}}$ be a sequence of mixing $\calz$-averse trees carried by $\tau$. Assume that $(T_n)$ converges to a mixing $\calz$-averse tree $T$, and that $\cala(\tau)$ is a maximal free factor system elliptic in $T$. For all $n\in\mathbb{N}$, let $f_n:S^\tau\to T_n$ be the carrying map, and let $f:S^\tau\to T$ be the limit of the maps $f_n$. 
\\ Let $S'$ be the tree obtained from $S^\tau$ by collapsing all edges whose $f$-image is reduced to a point, let $f':S'\to T$ be the induced map, and let $\tau':=\tau_{f'}$ be the train-track on $S'$ induced by $f'$. 
\\ Then $\tau'$ is admissible. Either $i_{geom}(\tau')>i_{geom}(\tau)$, or else $S'=S^\tau$ and $\tau'$ is obtained from $\tau$ by a finite (possibly trivial) sequence of specializations (hence $\tau$ carries $T$).
\end{prop}

Before proving Proposition~\ref{prop_index_up}, we first complete the proof of Proposition~\ref{closed} from the above facts.

\begin{proof}[Proof of Proposition \ref{closed}]
Let $\xi\in\partial P(\tau)$. If $\xi$ is carried by $\tau$, then $i(\xi)\ge i(\tau)$, and this inequality is strict because otherwise $\tau$ would be an ideal carrier of $\xi$, contradicting $\xi\in\partial P(\tau)$. 

We now assume that $\xi$ is not carried by $\tau$. Let $(\xi_n)_{n\in\mathbb{N}}\in P(\tau)^{\mathbb{N}}$ be a sequence converging to $\xi$, and for all $n\in\mathbb{N}$, let $T_n$ be a mixing representative of the equivalence class $\xi_n$. Since $\overline{cv}_N$ is projectively compact, up to a subsequence we can assume that $(T_n)_{n\in\mathbb{N}}$ converges to a tree $T$. Since the boundary map $\partial\pi:\calx_N\to\partial_\infty FZ_N$ is closed, the tree $T$ is $\calz$-averse, and in the $\approx$-class $\xi$. Using Lemma~\ref{lem_stab_up}, the proof reduces to the case where $T$ is mixing, and $\cala(\tau)$ is a maximal free factor system elliptic in $T$. Since $\xi$ is not carried by $\tau$, Proposition~\ref{prop_index_up} then shows that $T$ is carried by a train-track $\tau'$ satisfying $i_{geom}(\tau')>i_{geom}(\tau)$. This implies in turn that $i(\xi)>i(\tau)$.  
\end{proof}

The rest of the section is devoted to the proof of Proposition~\ref{prop_index_up}.

\begin{proof}[Proof of Proposition~\ref{prop_index_up}]
~\\\textbf{1. The track $\tau'$ is admissible.} 
Before proving that $\tau'$ is admissible, we first observe that all hyperbolic elements in $S^\tau$ are still hyperbolic in $S'$. Indeed, in view of Proposition~\ref{description-stab} and of our assumption that $\cala(\tau)$ is a maximal elliptic free factor system in $T$, if $g\in F_N\setminus\{e\}$ is elliptic in $T$ but is not contained in a free factor from $\cala(\tau)$, then the conjugacy class of $g$ is given by (some power of) a boundary curve. Hence $g$ is not contained in any proper free factor of $F_N$ relative to $\cala(\tau)$. Therefore the axis of $g$ in $S^\tau$ crosses all orbits of edges, and so cannot be collapsed to a point by the collapse map $\pi:S^\tau\to S'$.

We now prove that $\tau'$ is admissible. 
Let $v\in V(S')$, and let $X\subseteq S^\tau$ be the $\pi$-preimage of $v$. We first observe that $X$ is a bounded subtree of $S^\tau$: indeed, otherwise, we would find two oriented edges $e,e'$ in $X$ in the same $F_N$-orbit (say $e'=ge$) and pointing in the same direction; this would imply that $g$ is hyperbolic in $S^\tau$ but not in $S'$, a contradiction. Let $Y\subseteq X$ be a maximal legal subtree of $X$. Using the fact that $\tau$ is admissible, we can find three pairwise inequivalent edges $e_1,e_2,e_3$ lying outside of $X$ based at extremal vertices $v_1,v_2,v_3$ of $Y$ (these vertices are not necessarily distinct), such that the subtree $Y\cup e_1\cup e_2\cup e_3$ is legal in $\tau$. Furthermore, no edge $e_i$ is collapsed to a point by $\pi$. Since $\tau$ is admissible, Lemma~\ref{recurrent} gives three elements $g_1,g_2,g_3$, which act hyperbolically in $S^\tau$, whose axes are legal, and such that the axis of $g_i$ in $S^\tau$ crosses the segment $e_i\cup [v_i,v_{i+1}]\cup e_{i+1}$ (mod $3$). For all $n\in\mathbb{N}$, the map $f_n$ preserves alignment when restricted to each of these axes, so in the limit $f$ also preserves alignment when restricted to these axes. This implies that $g_1,g_2$ and $g_3$ are legal in $S'$. In addition, their axes form a legal tripod at $v$, so $\tau'$ is admissible.
\\
\\
\textbf{2. Controlling the index of $\tau'$.} We now prove that $i_{geom}(\tau')>i_{geom}(\tau)$ unless $\tau$ carries $T$. Given a vertex $v\in V(S')$, we will first establish that the contribution of the equivalence class $[v]$ to the index of $\tau'$ is no less than the sum of the contributions of its $\pi$-preimages in $V(S^\tau)$ to the index of $S^\tau$. 

Let $x:=f'(v)$. Let $X$ be the set of vertices in $S^\tau$ that are mapped to $x$ under $f$. As equivalent vertices in $\tau$ are mapped to the same point in $T$ under $f$, the set $X$ is a union of equivalence classes of vertices in $\tau$. As $\stab(x)=\stab(X)$, if two vertices in $X$ are in the same $F_N$-orbit, then they are in the same $\stab(x)$-orbit. Hence we may pick a finite set $E_1, \ldots, E_k$ of representatives of the $\stab(x)$-orbits of equivalence classes of vertices in $X$. Suppose that the images of these equivalence classes correspond to $l$ orbits of points in $T_n$  (where $l$ can be chosen to be independent of $n$ by passing to an appropriate subsequence). After reordering (and possibly passing to a further subsequence) we may assume $E_1, \ldots, E_l$  are mapped to distinct $\stab(x)$-orbits of points in $T_n$ for all $n\in\mathbb{N}$, and $E_{l+1},\ldots, E_k$ are exceptional classes in $\tau$ (notice that some of the classes $E_i$ with $i\in\{1,\dots,l\}$ might be exceptional as well). 

Suppose that each equivalence class $E_i$ has $\alpha_i$ gates in $\tau$ and has a stabilizer of rank $r_i$. As exceptional classes do not contribute to the geometric index, the amount $E_1,\dots,E_k$ contribute to the geometric index of $\tau$ is 
\begin{equation} 
\sum_{i=1}^{l} \alpha_i + 3\sum_{i=1}^l r_i -3l .
\end{equation} 
Let $E=[v]$ be the equivalence class corresponding to the image of $X$ in $\tau'$. Then $E$ contributes $\alpha' +3r' -3$ to the index of $\tau'$, where $\alpha'$ is the number of orbits of gates at $E$ and $r'$ is the rank of $\stab(x)$. If we define $\alpha=\sum_{i=1}^l \alpha_i$ and $r=\sum_{i=1}^l r_i$, then the difference between the index contribution of $E$ in $\tau'$ and the index contribution of $E_1, \ldots, E_k$ in $\tau$ is 
\begin{equation}\label{count} 
\alpha' - \alpha +3(r' -r) + 3(l-1). 
\end{equation} 

Our goal is to control the number of $\stab(x)$-orbits of gates lost when passing from $\tau$ to $\tau'$. For all $n\in\mathbb{N}$, we let $Y_n$ be the subtree of $T_n$ spanned by the points in $f_n(X)$. Since there are finitely many $\stab(x)$-orbits of equivalence classes of vertices in $X$, the tree $Y_n$ is obtained from the $\stab(x)$-minimal invariant subtree of $T_n$ by attaching finitely many orbits of finite trees (if $\stab(x)=\{e\}$, then $Y_n$ is a finite subtree of $T_n$). Recall that we assumed that $\cala(\tau)$ is a maximal elliptic free factor system in $T$. Using Proposition \ref{description-stab}, this implies that $\stab(x)$ is either cyclic or contained in $\cala(\tau)$ (it may be trivial). In the second case $\stab(x)$ is elliptic in $T_n$. Hence the quotient $G_n=Y_n/\stab(x)$ is a finite graph of rank $0$ or $1$. There are $l$ marked points in $G_n$ corresponding to the images of $E_1,\ldots, E_l$. 
\\
\\
\textbf{Claims:}
\begin{enumerate} \item If $e$ is an oriented edge in $S^\tau$ which is collapsed under $\pi$, and such that $f'(\pi(e))=x$, then the gate corresponding to $e$ in $\tau$ is mapped under $f_n$ to a direction in $Y_n$ at a point in $f_n(X)$.
\item If $e$ and $e'$ are oriented edges in $S^\tau$ that determine inequivalent directions $d,d'$ in $\tau$, based at vertices in $\stab(x).\{E_1,\dots,E_l\}$, and if $\pi(e)$ and $\pi(e')$ are nondegenerate and equivalent in $\tau'$, then for all sufficiently large $n\in\mathbb{N}$, one of the directions $d,d'$ is mapped under $f_n$ to a direction in $Y_n$ at a point in $f_n(X)$.
\end{enumerate}
\textbf{Proof of claims:} For the first claim, as $e$ is collapsed its endpoints lie in $X$, so that $f_n(e)$ is an arc between two points of $f_n(X)$ in $Y_n$. For the second claim, if $\pi(e)$ and $\pi(e')$ are equivalent, then $f(e)\cap f(e')$ is nondegenerate. Then for all sufficiently large $n\in\mathbb{N}$, the intersection of $f_n(e)$ with $f_n(e')$ contains a nondegenerate segment. As $e$ and $e'$ are not equivalent in $\tau$, they are based at distinct equivalence classes of vertices in $\stab(x).\{E_1,\dots,E_l\}$. Let us call these $[v_1]$ and $[v_2]$. We have $f_n(v_1)\neq f_n(v_2)$, and the arc connecting $f_n(v_1)$ with $f_n(v_2)$ is contained in the tree $Y_n$, and is covered by the union of the two arcs corresponding to $f_n(e)$ (starting at $f_n(v_1)$) and $f_n(e')$ (starting at $f_n(v_2)$). One can check that the initial direction of either $f_n(e)$ or $f_n(e')$ must then be contained in $Y_n$.
\\
\\
\indent Since $T$ has trivial arc stabilizers, no two oriented edges in the same orbit are equivalent in $\tau'$, and therefore there are only finitely many orbits of pairs of equivalent directions in $\tau'$. Therefore, we can choose $n\in\mathbb{N}$ large enough so that the conclusion of the second claim holds for all pairs of inequivalent directions in $S^\tau$ based at vertices in $\stab(x).\{E_1,\dots,E_l\}$, whose $\pi$-images are (nondegenerate and) equivalent in $S'$. 
\\
\indent Denote by  $\mathcal{G}_n^{out}$ (resp. $\mathcal{G}_n^{in}$) the set of gates based at vertices in $\stab(x).\{E_1,\dots,E_l\}$ which are mapped outside of $Y_n$ (resp. inside $Y_n$) by $f_n$. Claim 1 implies that there is a map $\Psi$ from the set $E_n^{out}$ of edges in $\mathcal{G}_n^{out}$, to the set of  gates at $E$ in $\tau'$. Claim~2 then implies that any two edges in $E_n^{out}$ in distinct gates have distinct $\Psi$-images, so $|\calg_n^{out}/\stab(x)|\le\alpha'$. Furthermore, our definition of $E_1,\dots,E_l$ implies that the natural map (induced by $f_n$) from $\mathcal{G}_n^{in}/\stab(x)$ to the set $\text{Dir}(G_n)$ of directions based at the marked points in the quotient graph $G_n$, is injective, so $|\calg_n^{in}/\stab(x)|\le |\text{Dir}(G_n)|$. By summing the above two inequalities, it follows that $\alpha-\alpha'$ is bounded above by the number of directions based at the marked points in the quotient graph $G_n$. 

We will now distinguish three cases. Notice that up to passing to a subsequence, we can assume that one of them occurs.
\\
\\
\textbf{Case 1:} The stabilizer $\text{Stab}(x)$ is either trivial, or fixes a  point $x_n$ in all trees $T_n$, and for all $n\in\mathbb{N}$, there is a vertex $v_n\in V(S^\tau)$ such that $f_n(v_n)=x_n$ (this happens in particular if $\stab(x)$ is contained in $\cala(\tau)$). 
\\ In this case, we have $r=r'$, and the graph $G_n$ has rank $0$. The $l$ marked points of $G_n$ include the leaves of $G_n$: indeed, every leaf of $G_n$ is either the projection to $G_n$ of a point in $f_n(X)$, or else it is the projection to $G_n$ of the unique point $y\in T_n$ with stabilizer equal to $\stab(x)$. In the latter case, we again have $y\in f_n(X)$ by assumption. We will apply the following fact to the graph $G_n$. 
\\
\\
\textbf{Fact:} If $G$ is a finite connected graph of rank $r$ with $l$ marked points containing all leaves of $G$, then there are at most $2(l+r-1)$ directions at these marked points.
\\
\\
\textbf{Proof of the fact:} An Euler characteristic argument shows that in any finite connected graph $G$ with $v$ vertices, there are exactly $2(v+r-1)$ directions at the vertices. Viewing the vertex set as a set of marked points, removing any non-leaf vertex from this set loses at least two directions, from which the fact follows.  
\\
\\
This fact applied to the quotient graph $G_n$ shows that 
\begin{equation}\label{gate} 
\alpha' - \alpha \geq -2(l-1). 
\end{equation} 
From \eqref{count} we see that the new equivalence class contributes at least $l-1$ more to the geometric index than the sum of the contributions of the previous equivalence classes to the index of $\tau$. This shows that the index increases as soon as $l\ge 2$, so we are left with the case where $l=1$. 

In this remaining situation, the graph $G_n$ is a single point, so all but possibly one (say $E_1$) of the equivalence classes $E_1,\dots,E_k$ are exceptional. By Claim~1 no edges corresponding to directions at $X$ are collapsed under $\pi$. Furthermore, by Claim~2 distinct equivalence classes of gates based at vertices in the orbit of $E_1$ are mapped to distinct gates under $\pi$. It follows that either there are more gates at the equivalence class $X$ when passing from $\tau$ to $\tau'$ (so the index increases), or the gates corresponding to directions at $E_1$ are mapped bijectively and each exceptional class in $E_2$, \ldots, $E_k$ corresponds to a specialization of $\tau$ (otherwise we would see extra gates). Repeating this argument across all equivalence classes in $\tau'$, we find that either the geometric index of $\tau'$ is greater than that of $\tau$, or $\tau'$ is obtained from $\tau$ by applying a finite number of specializations. In this latter case, this implies that $T$ is carried by $\tau$. 
\\
\\
\textbf{Case 2:} The group $\text{Stab}(x)$ is not elliptic in $T_n$ (in particular $\stab(x)$ is cyclic).
\\ In this case, we have $r'=1$ and $r=0$. The graph $G_n$ is then a circle, with finitely many finite trees attached, whose leaves are the projections of points in $f_n(X)$. In particular $G_n$ has rank $r'-r=1$, and all its leaves are marked. By applying the above fact, we thus get that 
\begin{equation}\label{gate2} 
\alpha' - \alpha \geq -2(l+r'-r-1). 
\end{equation} 
Since $r'-r=1$, we get from \eqref{count} that the new equivalence class contributes at least $l$ more to the geometric index than the sum of the contributions of the previous equivalence classes to the index of $\tau$, so $i(\tau')>i(\tau)$.
\\
\\
\textbf{Case 3:} The stabilizer $\text{Stab}(x)$ is cyclic, fixes a point $x_n$ in all trees $T_n$, but no vertex in $S^\tau$ is mapped to $x_n$ under $f_n$.
\\ In this case, we have $r=0$ and $r'=1$. The graph $G_n$ has rank $0$.  The image of $x_n$ in the quotient graph $G_n$ might not be marked, however the set of marked points in $G_n$ includes all other leaves. We will use the following variation on the above fact (which is easily proved by first including the missing leaf to the set of marked points). 
\\
\\
\textbf{Fact:} Suppose $G$ is a finite connected graph of rank $r$ with a set of $l$ marked points containing all leaves except possibly one. Then there are at most $2(l+r-1)+1$ directions at these marked points. 
\\
\\ This fact, applied to the graph $G_n$, shows that 
\begin{equation}\label{gate-2} 
\alpha' - \alpha \geq -2(l-1)-1. 
\end{equation}
Since $r'-r=1$, we get from \eqref{count} that the new equivalence class contributes at least $l+1$ more to the geometric index than the sum of the contributions of the previous equivalence classes to the index of $\tau$, and again we are done.
\end{proof}

\section{The cells $P(\tau)$ have dimension at most $0$.}\label{sec-dim}

The goal of this section is to prove the following fact.

\begin{prop}\label{dim-0}
Let $\tau$ be an admissible train-track. Then $P(\tau)$ has dimension at most $0$.
\end{prop}

The strategy of our proof of Proposition~\ref{dim-0} is the following. Given any point $\xi\in P(\tau)$, our goal is to construct arbitrarily small open neighborhoods of $\xi$ with empty boundary in $P(\tau)$. This will be done using folding sequences of train-tracks. In Section~\ref{sec-strat}, we defined the notion of a specialization, which gives a new train-track $\tau'$ from a train-track $\tau$. In Section \ref{sec-moves}, we will introduce other operations called \emph{folding moves}. These will enable us to define a new train-track $\tau'$ from a train-track $\tau$ by folding at an illegal turn. We will then make the following definition.

\begin{de}[\textbf{\emph{Folding sequence of train-tracks}}]\label{de-fsott}
A \emph{folding sequence of train-tracks} is an infinite sequence $(\tau_i)_{i\in\mathbb{N}}$ of admissible train-tracks such that for all $i\in\mathbb{N}$, the train-track $\tau_{i+1}$ is obtained from $\tau_i$ by applying a folding move followed by a finite (possibly trivial) sequence of specializations.
\\ Given $\xi\in\partial_\infty FZ_N$, we say that the folding sequence of train-tracks $(\tau_i)_{i\in\mathbb{N}}$ is \emph{directed by $\xi$} if $\xi\in P(\tau_i)$ for all $i\in\mathbb{N}$. 
\end{de}

We will first prove in Section~\ref{sec-moves} that folding sequences of train-tracks exist.

\begin{lemma}\label{fold-sequence}
Let $\tau$ be an admissible train-track, and let $\xi\in P(\tau)$. 
\\ Then there exists a folding sequence of train-tracks $(\tau_i)_{i\in\mathbb{N}}$ directed by $\xi$, such that $\tau_0$ is obtained from $\tau$ by a finite (possibly trivial) sequence of specializations.
\end{lemma}

We will then prove in Section \ref{sec-open} that all sets $P(\tau_i)$ in a folding sequence of train-tracks are open in $P(\tau_0)$, by showing the following two facts.

\begin{lemma}\label{specialization}
Let $\tau$ be an admissible train-track, and let $\tau'$ be a specialization of $\tau$. Then $\tau'$ is admissible, and $P(\tau')$ is an open subset of $P(\tau)$.
\end{lemma}

\begin{lemma}\label{fold-open}
Let $\tau$ be an admissible train-track, and let $\tau'$ be a train-track obtained from $\tau$ by folding an illegal turn. Then $\tau'$ is admissible, and $P(\tau')$ is an open subset of $P(\tau)$. 
\end{lemma}   

In other words, the sets $P(\tau_i)$ are open neighborhoods of $\xi$ in $P(\tau)$, and they are closed in $P(\tau)$ (because their boundaries are made of trees with higher index in view of Proposition~\ref{closed}). So to complete the proof of Proposition~\ref{dim-0}, we are left showing that $P(\tau_i)$ can be made arbitrary small. This is proved in Section~\ref{sec-diam} in the form of the following proposition.

\begin{lemma}\label{fold-diam}
Let $\tau$ be an admissible train-track, let $\xi\in P(\tau)$, and let $(\tau_i)_{i\in\mathbb{N}}$ be a folding sequence of train-tracks directed by $\xi$. 
\\ Then the diameter of $P(\tau_i)$ converges to $0$.
\end{lemma}

We now sum up the proof of Proposition~\ref{dim-0} from the above four lemmas.

\begin{proof}[Proof of Proposition~\ref{dim-0}]
Assume that $P(\tau)\neq\emptyset$, and let $\xi \in P(\tau)$. Let $\epsilon >0$. Let $(\tau_i)_{i\in\mathbb{N}}$ be a folding sequence of train-tracks directed by $\xi$  provided by Lemma~\ref{fold-sequence}, where  $\tau_0$ is obtained from $\tau$ by a finite (possibly trivial) sequence of specializations. Lemma~\ref{fold-diam} shows that we can find $k\in\mathbb{N}$ such that $P(\tau_k)$ contains $\xi$ and has diameter at most $\epsilon$. An iterative application of Lemmas~\ref{specialization} and~\ref{fold-open} then ensures that $P(\tau_k)$ is an open neighborhood of $\xi$ in $P(\tau)$. Furthermore, the boundary of $P(\tau_k)$ in $P(\tau)$ is empty (it is made of trees of higher index by Proposition~\ref{closed}). As each point in $P(\tau)$ has arbitrarily small open neighborhoods with empty boundary, $\dim(P(\tau))=0$.
\end{proof}

\subsection{Existence of folding sequences of train-tracks directed by a boundary point}\label{sec-moves}

\subsubsection{More on specializations}

The notion of specialization was given in Definition~\ref{de-specialization}. 

\begin{lemma}\label{lem_spec}
Let $\tau$ be an admissible train-track, and let $\tau'$ be a specialization of $\tau$.
\\ Then $\tau'$ is admissible, and $i(\tau')=i(\tau)$.
\end{lemma}

\begin{proof}
Admissibility is clear, as the underlying tree $S^\tau$ and the collection of legal turns at each vertex of $S^\tau$ are unchanged after a specialization. To see that $i(\tau')=i(\tau)$, notice first that the exceptional vertices in $S^\tau$ contribute $0$ to the geometric index of $\tau$. Stabilizers do not increase by definition of a specialization (otherwise two inequivalent vertices in $\tau$ that are not in the orbit of $v_0$ would become equivalent in $\tau'$), and no new gate is created. So $i_{geom}(\tau')=i_{geom}(\tau)$, and since stabilizers of equivalence classes of vertices are the same in $\tau$ and in $\tau'$, we have $i(\tau')=i(\tau)$.
\end{proof}

\subsubsection{Folding moves}

We now introduce three types of folding moves and discuss some of their properties. 

\begin{de}[\emph{\textbf{Singular fold, see Figure \ref{fig-singular}}}] \label{d:singular}
Let $\tau$ be a train-track. Let $e_1=[v,v_1]$ and $e_2=[v,v_2]$ be two edges in $S^\tau$ that are based at a common vertex $v$, determine equivalent directions at $v$, and such that $v_1\sim^\tau v_2$. 
\\ A train-track $\tau'$ is obtained from $\tau$ by a \emph{singular fold} at $\{e_1,e_2\}$ if 
\begin{itemize}
\item there is an $F_N$-equivariant map $g:S^{\tau}\to S^{\tau'}$ that consists in equivariantly identifying $e_1$ with $e_2$,
\item for all vertices $v,v'\in V(S^\tau)$, we have $g(v')\sim^{\tau'}g(v)$ if and only if $v'\sim^{\tau}v$,
\item for all directions $d,d'$ in $S^{\tau}$ based at equivalent vertices of $S^\tau$, we have $g(d')\sim^{\tau'} g(d)$ if and only if $d'\sim^{\tau}d$.
\end{itemize}
\end{de}

\begin{rk}\label{r:singular_exists}
If $\tau$ satisfies the hypothesis of Definition~\ref{d:singular} with edges $e_1$ and $e_2$, then such a singular fold $\tau'$ at $\{e_1,e_2\}$ exists and is unique as long as the tree $S'$ obtained by folding $S^\tau$ along $e_1$ and $e_2$ has trivial edge stabilizers. As any carrying map factors through the folding map $g:S^\tau \to S'$, a sufficient condition for the existence of $\tau'$ is that $\tau$ carries a tree with trivial arc stabilizers. This covers all of the relevant train-tracks in this paper. 
\end{rk}


\begin{figure}
\begin{center}
\input{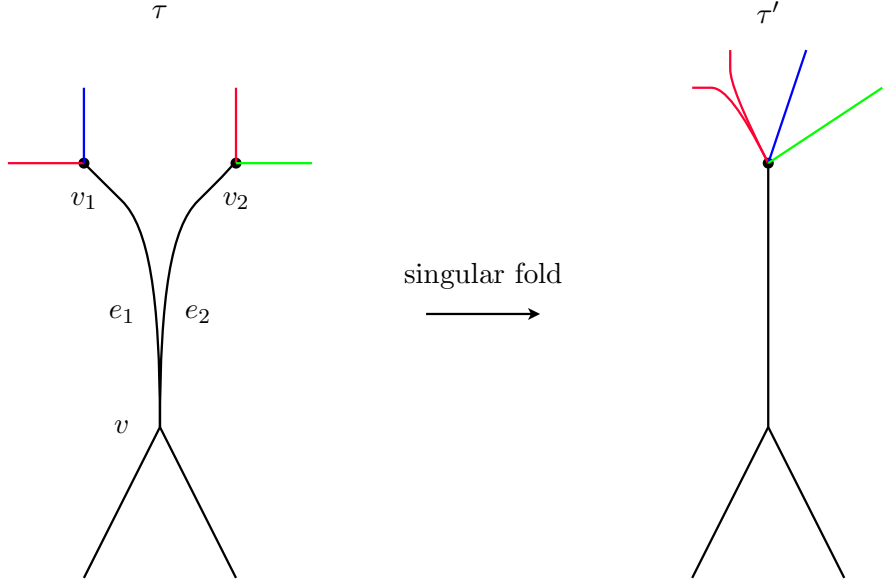}
\caption{A singular fold. The two black vertices $v_1$ and $v_2$ are equivalent in $\tau$, and the colors give the gates at the equivalence class of vertices they define.}
\label{fig-singular}
\end{center}
\end{figure}

\begin{de}\label{de-p}
Given an admissible train-track $\tau$, we denote by $\calp(\tau)$ the subspace of $\partial{cv}_N$ made of trees $T$ with dense orbits such that $\tau$ is an ideal carrier of $T$.  
\end{de}

In particular, trees in $\calp(\tau)$ have trivial arc stabilizers. 

\begin{lemma}\label{lem_singular}
Let $\tau$ be an admissible train-track, and let $e_1=[v,v_1]$ and $e_2=[v,v_2]$ be two edges that form an illegal turn in $\tau$, such that $v_1\sim^\tau v_2$, and such that the directions at $v_1$ and $v_2$ pointing towards $v$ are $\tau$-equivalent. 
\\ If $\tau$ carries a tree with trivial arc stabilizers then there exists an admissible train-track $\tau'$ obtained from $\tau$ by a singular fold at $\{e_1,e_2\}$ such that $\calp(\tau')=\calp(\tau)$.
\end{lemma}

\begin{proof}
As $\tau$ carries a tree with trivial arc stabilizers, following Remark~\ref{r:singular_exists} there exists a unique train track $\tau'$ obtained by a singular fold at $\{e_1,e_2\}$. Let $g:S^{\tau} \to S^{\tau'}$ be the folding map given in Definition~\ref{d:singular}. The numbers and the stabilizers of equivalence classes of vertices and gates are unchanged under $g$, so that $i(\tau')=i(\tau)$. To check admissibility of $\tau'$, let $v'$ be a vertex in $S^{\tau'}$. Then there exists $v\in V(S^\tau)$ such that $v'=g(v)$, and the $g$-image of any tripod of legal axes at $v$ is a tripod of legal axes at $v'$ (because $g$ sends legal turns to legal turns). 

We finally show that $\calp(\tau)=\calp(\tau')$. Let $T\in\calp(\tau)$, with carrying map $f:S^\tau \to T$. Since the extremities of $e_1$ and $e_2$ are equivalent, these two edges are mapped to the same segment in $T$. It follows that the carrying map $f$ factors through the fold $g$ to attain a map $f':S^{\tau'} \to T$, and one can check that the train-track induced by $f'$ is obtained from $\tau'$ by a finite (possibly trivial) sequence of specializations (using the same sequence of specializations as when passing from $\tau$ to $\tau_f$). So $T$ is carried by $\tau'$, and equality of the indices of $\tau$ and $\tau'$ shows that $\tau'$ is also an ideal carrier of $T$. It follows that $\calp(\tau) \subseteq \calp(\tau')$. Conversely, if $\tau'$ is an ideal carrier of $T$ with carrying map $f'$ then the composition $f=f' \circ g$ is such that $\tau_f$ is obtained from $\tau$ by a finite (possibly trivial) sequence of specializations, and it follows that $\calp(\tau') \subseteq \calp(\tau)$.
\end{proof}

\begin{de}[\emph{\textbf{Partial fold, see Figure \ref{fig-partial}}}]\label{de-partial}
Let $\tau$ be a train-track, and let $\{e_1,e_2\}$ be an illegal turn at a vertex $v\in S^\tau$. A train-track $\tau'$ is obtained from $\tau$ by a \emph{partial fold} at $\{e_1,e_2\}$ if 
\begin{itemize}
\item there is a map $g:S^{\tau}\to S^{\tau'}$ that consists in equivariantly identifying a proper initial segment $[v,v'_1]$ of $e_1$ with a proper initial segment $[v,v'_2]$ of $e_2$, so that the vertex $v'=g(v'_1)=g(v'_2)$ is trivalent in $S^{\tau'}$ (all vertices of $S^{\tau'}$ that are not in the $F_N$-orbit of $v'$ have a unique $g$-preimage in $S^{\tau}$),  
\item for all $v,w\in V(S^\tau)$, we have $g(v)\sim_V^{\tau'} g(w)$ if and only if $v\sim_V^\tau w$, 
\item for all directions $d,d'$ based at vertices of $S^\tau$, we have $g(d')\sim^{\tau'} g(d)$ if and only if $d'\sim^{\tau}d$,
\item the vertex $v'$ is not equivalent to any other vertex in $\tau'$, and all directions at $v'$ in $S^{\tau'}$ are pairwise inequivalent.
\end{itemize}
\end{de}

\begin{rk}
The last condition in the definition implies that the new vertex $v'$ is exceptional in $\tau'$. As with singular folds, such a partial fold $\tau'$ at $\{e_1,e_2\}$ exists and is unique as long as partially folding $S^\tau$ along $e_1$ and $e_2$ gives a tree with trivial arc stabilizers, which will be the case when $\calp(\tau)$ is nonempty.  \end{rk}

\begin{figure}
\begin{center}
\input{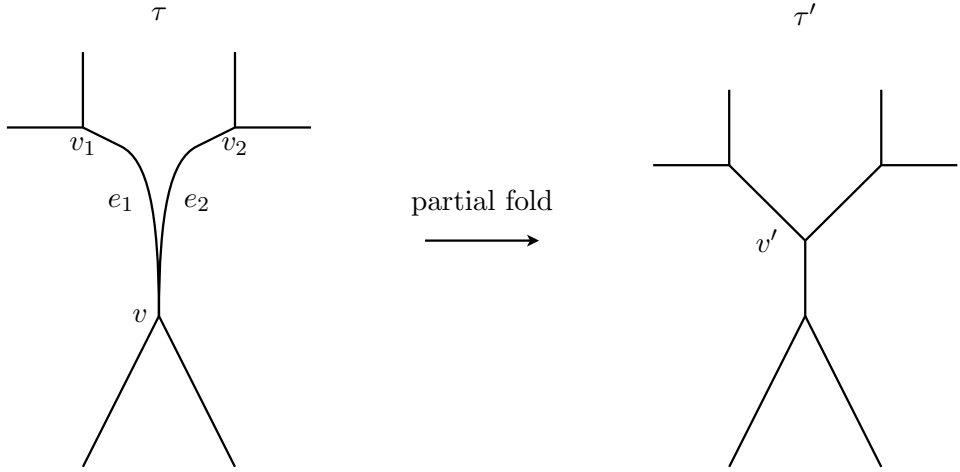}
\caption{A partial fold.}
\label{fig-partial}
\end{center}
\end{figure}

\begin{lemma}\label{lem_partial}
Let $\tau$ be an admissible train-track, and let $\tau'$ be a train-track obtained from $\tau$ by a partial fold.
\\ Then $\tau'$ is admissible, and $i(\tau')=i(\tau)$.
\end{lemma}

\begin{proof}
The condition from the definition of admissibility is easy to check at any vertex in $S^{\tau'}$ not in the orbit of $v'$ (with the notations from Definition \ref{de-partial}). Let $d_0,d_1,d_2$ be the three directions at $v'$, containing the $g$-images of $v=v_0,v_1,v_2$, respectively. 
Let $i\in\{0,1,2\}$ (considered modulo $3$). To construct a legal element whose axis crosses the turn $\{d_i,d_{i+1}\}$, we argue as in the proof of Lemma~\ref{recurrent}: one first finds two elements $h,h'$ that are legal in both $\tau$ and $\tau'$ such that \begin{itemize} \item the axes $A_h,A_{h'}$ in $S^{\tau'}$ go through $g(v_i)$ and $g(v_{i+1})$ respectively, \item the axes $A_h$ and $A_{h'}$ do not go through $v'$, \item the subtree $Y$ spanned by $A_h$ and $A_{h'}$ (which contains the turn $\{d_i,d_{i+1}\}$) is legal. \end{itemize} The element $hh'$ is then a legal element that crosses the turn $\{d_i,d_{i+1}\}$, as required.   

To see that $i(\tau')=i(\tau)$, we first note that $\cala(\tau')=\cala(\tau)$ because the stabilizer of any equivalence class corresponding to a new vertex is trivial, and the stabilizers of the other equivalence classes of vertices are unchanged when passing from $\tau$ to $\tau'$. In addition, under the folding process the number of nonexceptional equivalence classes of vertices, along with the ranks of their stabilizers and the number of associated gates, remains the same, and the new exceptional trivalent vertex contributes $0$ to the index. Hence $i_{geom}(\tau')=i_{geom}(\tau)$, and therefore $i(\tau')=i(\tau)$ also. 
\end{proof}

\begin{lemma}\label{fold-p}
Let $\tau$ be an admissible train-track, and let $T\in\calp(\tau)$ with carrying map $f:S^\tau\to T$. Assume that there exist two edges $e_1,e_2$ in $S^\tau$ that form an illegal turn, such that $|f(e_1)\cap f(e_2)|$ is smaller than both $|f(e_1)|$ and $|f(e_2)|$. 
\\ Then there exists a train-track $\tau'$ obtained from $\tau$ by applying a partial fold at $\{e_1,e_2\}$ such that $T\in\calp(\tau')$. 
\end{lemma}

\begin{proof}
Recall that $S^{\tau}$ is equipped with the metric induced by the map $f:S^\tau \to T$, so that any edge $e$ of $S^{\tau}$ has length $|f(e)|$. Let $S'$ be the simplicial tree obtained from $S^{\tau}$ by equivariantly identifying a proper initial segment $[v,v'_1]$ of $e_1$ of length $|f(e_1)\cap f(e_2)|$ with a proper initial segment $[v,v'_2]$ of $e_2$ of the same length. Let $g:S^\tau\to S'$ be the folding map. We first claim that the vertex $v':=g(v'_1)=g(v'_2)$ is trivalent in $S'$. Indeed, otherwise, the vertex $v'$ would have infinite stabilizer, so $\cala(\tau)$ would not be a maximal free factor system elliptic in $T$, a contradiction. We thus have $S'=S^{\tau'}$. The map $f$ factors through a map $f':S^{\tau'}\to T$. Since $i(\tau')=i(\tau)$, it is enough to prove that the train-track $\tau_{f'}$ is obtained from $\tau'$ by a finite sequence of specializations. 

By definition, the train-track $\tau_f$ is obtained from $\tau$ by a finite sequence of specializations. Let $(\tau_f)'$ be the partial fold of $\tau_f$ at $\{e_1,e_2\}$, and notice that $(\tau_f)'$ is obtained from $\tau'$ by the same sequence of specializations as when passing from $\tau$ to $\tau_f$.

If the $f'$-image of the new trivalent vertex $v'$ is not equal to the $f$-image of any other vertex in $S^\tau$, then $\tau_{f'}=(\tau_f)'$ and we are done. Otherwise, we will show that $\tau_{f'}$ is a specialization of $(\tau_f)'$, from which the lemma will follow. Indeed, if $f'(v')$ is equal to the $f$-image of another vertex $v_0\in S^\tau$, we first observe that we can always find such a $v_0$ which is not in the orbit of $v'$. This is because if $v'$ is only identified with vertices in its own orbit, then the equivalence class of $v'$ in $\tau_{f'}$ has nontrivial stabilizer in $\tau_{f'}$, and contributes positively to the geometric index of $\tau_{f'}$. Then $i(\tau_{f'})>i(\tau)$, contradicting the fact that $\tau$ is an ideal carrier. So we can assume that $v_0$ and $v'$ are not in the same orbit. Next, we observe that all germs of directions at $v'$ are identified by $f'$ with germs of directions at vertices that are $\tau$-equivalent to $v_0$, as otherwise we would be creating a new gate when passing from $\tau$ to $\tau_{f'}$, again increasing the index. So $\tau_{f'}$ is a specialization of $(\tau_f)'$.
\end{proof}

\begin{de}[\emph{\textbf{Full fold, see Figure \ref{fig-full}}}]\label{de-full}
Let $\tau$ be a train-track, and let $\{e_1,e_2\}$ be an illegal turn at a vertex $v\in S^\tau$. Denote by $v_1,v_2$ the other extremities of $e_1,e_2$. Let $d_0$ be a direction in $\tau$ based at a vertex $\tau$-equivalent to $v_1$. A train-track $\tau'$ is obtained from $\tau$ by a \emph{full fold of $e_1$ into $e_2$ with special gate $[d_0]$} if 
\begin{itemize}
\item there is an $F_N$-equivariant map $g:S^{\tau}\to S^{\tau'}$ that consists in equivariantly identifying $[v,v_1]$ with a proper initial segment $[v,v'_2]\subsetneq[v,v_2]$,
\item for all vertices $v,v'\in V(S^\tau)$, we have $g(v')\sim^{\tau'}g(v)$ if and only if $v'\sim^{\tau}v$,
\item for all directions $d,d'$ based at vertices of $S^\tau$, we have $g(d')\sim^{\tau'}g(d)$ if and only if $d'\sim^{\tau}d$,
\item if $d$ is the direction in $S^{\tau'}$ based at $g(v_1)=g(v'_2)$ and pointing towards $g(v_2)$, then $d\sim g(d_0)$.
\end{itemize}
\end{de}

\begin{figure}
\begin{center}
\input{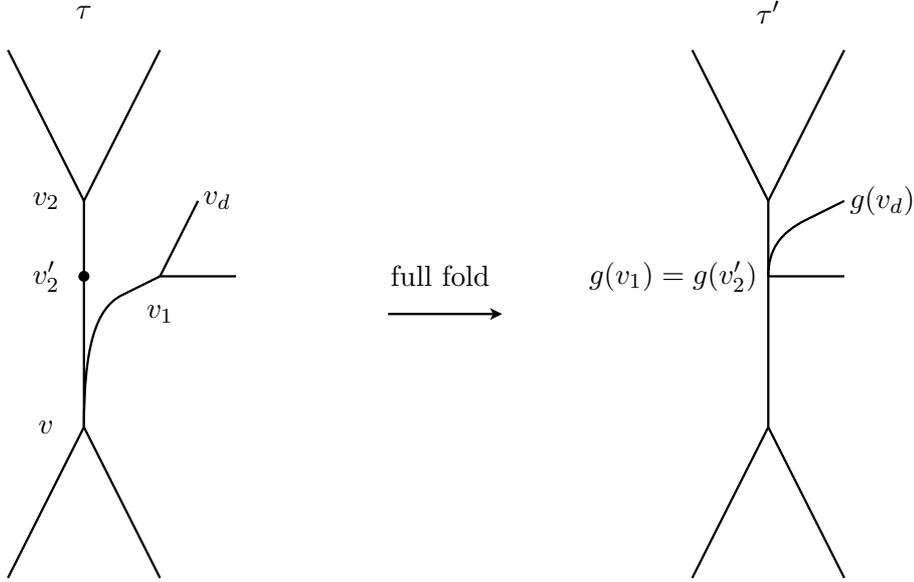}
\caption{A full fold. The new direction at $g(v_1)=g(v'_2)$ has been identified with the $g$-image of the direction $d_0$ at $v_1$ with initial edge $[v_1,v_d]$ (in general, the direction $d_0$ may also be based at a vertex $v'_1\neq v_1$ as long as $v_1$ and $v'_1$ are $\tau$-equivalent).}
\label{fig-full}
\end{center}
\end{figure}

The number of equivalence classes of vertices remains unchanged under a full fold, and the vertex stabilizers in $S^{\tau'}$ are the same as the vertex stabilizers in $S^\tau$. The fold creates no new gates, so the index remains unchanged. If $\tau$ is admissible then $\tau'$ is also admissible. Hence:

\begin{lemma}\label{lem_full}
Let $\tau$ be an admissible train-track, and let $\tau'$ be a train-track obtained from $\tau$ by a full fold.
\\ Then $\tau'$ is admissible and $i(\tau')=i(\tau)$.
\qed
\end{lemma}

\begin{lemma}\label{fold-f-1}
Let $\tau$ be an admissible train-track, and let $T\in\calp(\tau)$ with carrying map $f:S^\tau\to T$. Let $e_1=[v,v_1]$ and $e_2=[v,v_2]$ be two edges in $S^\tau$ that form an illegal turn. Assume that $f(e_1)\subsetneq f(e_2)$. Also assume that there is a direction $d_0$ (whose initial edge we denote by $[v'_1,v_d]$) at a vertex $v'_1 \sim^\tau v_1$ such that 
\begin{equation}\label{eq<}
d_T(f(v_d),f(v_2))<d_T(f(v_d),f(v'_1))+d_T(f(v'_1),f(v_2)).
\end{equation}
 \\ Then there exists a train-track $\tau'$ obtained from $\tau$ by fully folding $e_1$ into $e_2$, with special gate $[d_0]$ such that $T \in \calp(\tau')$.
\end{lemma}

\begin{proof}
As in the previous cases, existence of $\tau'$ follows from the fact that $\tau$ carries a tree with trivial arc stabilizers. Since $f(e_1)\subsetneq f(e_2)$, the map $f$ factors through the fold $g:S^{\tau}\to S^{\tau'}$ to reach an $F_N$-equivariant map $f':S^{\tau'}\to T$, and Equation~\eqref{eq<} ensures that $f'$ identifies germs of $[g(v_1),g(v_2)]$ and of $g(d_0)$. So $\tau_{f'}$ is obtained from $\tau'$ by the same finite sequence of specializations as when passing from $\tau$ to $\tau_f$. We know in addition that $i(\tau')=i(\tau)$ (Lemma~\ref{lem_full}). This implies that $T\in\calp(\tau')$.
\end{proof}

\begin{cor}\label{fold-f}
Let $\tau$ be an admissible train-track, and let $T\in\calp(\tau)$, with carrying map $f:S^\tau\to T$. Let $e_1=[v,v_1]$ and $e_2=[v,v_2]$ be two edges in $S^\tau$ that form an illegal turn. Assume that $f(e_1)\subsetneq f(e_2)$. Also assume that $T$ is not carried by any specialization of $\tau$.
\\ Then there exists a train-track $\tau'$ obtained from $\tau$ by fully folding $e_1$ into $e_2$, such that $T\in\calp(\tau')$.
\end{cor}

\begin{proof}
Let $S'$ be the underlying simplicial tree of any full fold of $\tau$. Since $f(e_1)\subsetneq f(e_2)$, the map $f$ factors through an $F_N$-equivariant map $f':S^{\tau'}\to T$. Equality in indices shows the existence of a direction $d_0$ in $\tau$ for which Equation~\eqref{eq<} from Lemma~\ref{fold-f-1} holds, as otherwise the fold would create a new gate, and $T$ would be carried by a train-track with higher index. The direction $d_0$ is based at a vertex $v'_1$ satisfying $f(v'_1)=f(v_1)$. Since $T$ is not carried by any specialization of $\tau$, this implies that $v'_1$ is $\tau$-equivalent to $v_1$. Therefore, Lemma~\ref{fold-f-1} implies that $T\in\calp(\tau')$, where $\tau'$ is the full fold of $\tau$ with special gate $[d_0]$.
\end{proof}

\subsubsection{Folding sequences: proof of Lemma~\ref{fold-sequence}}

Given two train-tracks $\tau$ and $\tau'$, we say that $\tau'$ is obtained from $\tau$ by \emph{applying a folding move} if $\tau'$ is obtained from $\tau$ by applying either a singular fold, a partial fold or a full fold. We then define a \emph{folding sequence of train-tracks} as in Definition~\ref{de-fsott}: recall that this is a sequence $(\tau_i)_{i\in\mathbb{N}}$ of train-tracks such that $\tau_{i+1}$ is obtained from $\tau_i$ by applying a folding move, followed by a finite (possibly trivial) sequence of specializations. The goal of the present section is to prove Lemma~\ref{fold-sequence}. We will actually prove a slightly stronger version of it, given in Lemma~\ref{strong} below. The main fact we will use in the proof is the following.

\begin{lemma}\label{fold-exist}
Let $\tau$ be an admissible train-track, and let $T\in\mathcal{P}(\tau)$.
\\ Then there exists a train-track $\tau'\neq\tau$ obtained either by a specialization, or by applying a folding move, such that $T\in\mathcal{P}(\tau')$. 
\end{lemma}

\begin{proof}
Assume that $T$ is not carried by any train-track $\tau'\neq\tau$ obtained from $\tau$ by performing a specialization. Let $f:S^\tau \to T$ be the carrying map, and let $e_1=[v,v_1]$ and $e_2=[v,v_2]$ be two edges that form an illegal turn in $\tau$: this exists because $T$ is not simplicial by definition of $\calp(\tau)$. If $f(v_1)=f(v_2)$, then $v_1\sim^{\tau}v_2$, as otherwise $T$ would be carried by a specialization of $\tau$. In this case $f$ identifies germs of the directions at $v_1$ and $v_2$ pointing towards $v$, so Lemma~\ref{lem_singular} implies that $T\in\calp(\tau')$, where $\tau'$ is the singular fold of $\tau$. If $f(v_1)\neq f(v_2)$, and if $|f(e_1)\cap f(e_2)|$ is smaller than both $|f(e_1)|$ and $|f(e_2)|$, then Lemma~\ref{fold-p} implies that $T$ is carried by the partial fold of $\tau$ at $\{e_1,e_2\}$. Finally, if $f(e_1)\subsetneq f(e_2)$ (or vice versa), then Corollary~\ref{fold-f} implies that $T$ is carried by a full fold of $\tau$.  
\end{proof}

We also make the following observation.

\begin{lemma}\label{subspace}
Let $\tau$ be an admissible train-track, and let $\tau'$ be a train-track obtained from $\tau$ by applying either a specialization or a folding move.
\\ Then $\calp(\tau')\subseteq\calp(\tau)$ (and hence $P(\tau')\subseteq P(\tau)$). 
\end{lemma}

\begin{proof}
If $\tau'$ is obtained from $\tau$ by applying a specialization, then the conclusion follows from the definition of being carried, together with the fact that $i(\tau')=i(\tau)$ (Lemma~\ref{lem_spec}). If $\tau'$ is obtained from $\tau$ by applying a folding move, then the conclusion follows from the fact that $i(\tau')=i(\tau)$ (Lemmas~\ref{lem_singular},~\ref{lem_partial} and~\ref{lem_full}), together with the observation that if $g:S^\tau\to S^{\tau'}$ is the fold map, and $f:S^{\tau'}\to T$ is a carrying map, then $f\circ g:S^\tau\to T$ is also a carrying map.
\end{proof}

We are now in position to prove the following stronger version of Lemma~\ref{fold-sequence}.

\begin{lemma}\label{strong}
Let $\xi\in\partial_\infty FZ_N$. Let $\tau_0,\dots,\tau_k$ be a finite sequence of admissible train-tracks such that $\xi\in P(\tau_k)$, and for each $i\in\{0,\dots,k-1\}$, the train-track $\tau_{i+1}$ is obtained from $\tau_i$ by applying a folding move followed by a finite (possibly trivial) sequence of specializations.
\\ Then there exists a folding sequence $(\tau'_i)_{i\in\mathbb{N}}$ of train-tracks directed by $\xi$, such that $\tau_i=\tau'_i$ for all $i\in\{0,\dots,k-1\}$, and  $\tau'_k$ is obtained from $\tau_k$ by a finite (possibly trivial) sequence of specializations.
\end{lemma}

\begin{proof}
Lemma \ref{subspace} shows that $\xi\in P(\tau_i)$ for all $i\in\{1,\dots,k\}$. The conclusion is then obtained by iteratively applying Lemma~\ref{fold-exist} to some mixing representative of the class $\xi$ (starting with the train-track $\tau_k$), and noticing that we can only perform finitely many specializations in a row.
\end{proof}

\subsection{Openness of the set of points carried by a specialization or a fold}\label{sec-open}

\subsubsection{From trees to equivalence classes of trees}

We first reduce the proofs of Lemmas~\ref{specialization} and~\ref{fold-open} to their analogous versions for trees in $\overline{cv}_N$.

\begin{lemma}\label{trees-to-equiv}
Let $\tau$ and $\tau'$ be admissible train-tracks. 
\\ If $\calp(\tau')$ is an open subset of $\calp(\tau)$ in $\partial{cv}_N$, then $P(\tau')$ is an open subset of $P(\tau)$ in $\calx_N/{\approx}$.
\end{lemma}

\begin{proof}
The conclusion is obvious if $\calp(\tau')$ is empty (in this case $P(\tau')$ is also empty), so we assume otherwise. Since $\calp(\tau')\subseteq\calp(\tau)$, we have $i(\tau')=i(\tau)$. Let $\xi\in P(\tau')$, and let $(\xi_n)_{n\in\mathbb{N}}\in P(\tau)^{\mathbb{N}}$ be a sequence converging to $\xi$. We wish to prove that $\xi_n\in P(\tau')$ for all sufficiently large $n\in\mathbb{N}$ (we may use sequential arguments as $\calx_N/{\approx}$ is a separable metric space).

For all $n\in\mathbb{N}$, let $T_n$ be a mixing representative of $\xi_n$. Since $\overline{cv}_N$ is projectively compact, up to passing to a subsequence, we can assume that $(T_n)_{n\in\mathbb{N}}$ converges  projectively to a tree $T\in\overline{cv}_N$. Closedness of the boundary map $\partial\pi:\calx_N\to\partial_\infty FZ_N$ shows that $T\in\calx_N$, and $T$ is a representative of the class $\xi$. Since $\cala(\tau)$ is elliptic in each tree $T_n$, it is also elliptic in $T$. If $T$ is not mixing, then Corollary~\ref{cor-mixing} implies that $\cala(\tau)$ is not the largest free factor system elliptic in the mixing representatives of $\xi$, contradicting that $i(\xi)=i(\tau)$. So $T$ is mixing, and the fact that $\xi\in P(\tau')$ thus implies that $T\in\calp(\tau')$. Openness of $\calp(\tau')$ in $\calp(\tau)$ then shows that $T_n\in\calp(\tau')$ for all sufficiently large $n\in\mathbb{N}$. Since $T_n$ is a mixing representative of $\xi_n$, this precisely means that $\xi_n\in P(\tau')$ for all sufficiently large $n\in\mathbb{N}$, as desired.
\end{proof}

\subsubsection{Specializations}\label{sec-specialize}

In this section, we will prove Lemma~\ref{specialization}. More precisely, we will prove its analogous version for trees in $\overline{cv}_N$, from which Lemma~\ref{specialization} follows thanks to Lemma~\ref{trees-to-equiv}.


\begin{lemma}\label{special}
Let $\tau$ be an admissible train-track, and let $\tau'$ be a specialization of $\tau$.
\\ Then $\calp(\tau')$ is an open subset of $\calp(\tau)$.
\end{lemma}

\begin{proof}
By Lemma~\ref{subspace}, we have $\calp(\tau')\subseteq\calp(\tau)$. Let $T\in\calp(\tau')$, and let $(T_n)_{n\in\mathbb{N}}\in \calp(\tau)^{\mathbb{N}}$ be a sequence of trees that converges to $T$. We aim to show that $T_n$ is carried by $\tau'$ for all sufficiently large $n\in\mathbb{N}$, which will imply that $T_n\in\calp(\tau')$ by equality of the indices. Let $[v_0]$ be the class of vertices in $S^\tau$ at which the specialization occurs, and let $v_1\in V(S^\tau)$ be a vertex identified with $v_0$ in $\tau'$.

Let $e_0^1,e_0^2,e_0^3$ be three edges adjacent to vertices in the class $[v_0]$, determining distinct gates, and let $e_1^1,e_1^2,e_1^3$ be edges based at vertices $\tau$-equivalent to $v_1$ that are identified with $e_0^1,e_0^2,e_0^3$ in $\tau'$. We aim to show that for all sufficiently large $n\in\mathbb{N}$, the vertices $v_0$ and $v_1$ are identified by the carrying map $f_n:S^\tau\to T_n$, and so are initial segments of $e_0^i$ and $e_1^i$ for all $i\in\{1,2,3\}$. 

For all $i\in\{1,2,3\}$, the intersection $f(e_0^i)\cap f(e_1^i)$ (where $f:S^\tau\to T$ is the carrying map) is nondegenerate. Since the carrying map varies continuously with the carried tree (Lemma~\ref{carrier-unique}), for all sufficiently large $n\in\mathbb{N}$, the image $f_n(e_1^i)$ has nondegenerate intersection with $f_n(e_0^i)$. Since the three directions determined by $e_0^1,e_0^2,e_0^3$ are inequivalent, their images $f_n(e_0^1),f_n(e_0^2),f_n(e_0^3)$ form a tripod at $f_n(v_0)$. If $f_n(v_1)\neq f_n(v_0)$, then at least two of the images $f_n(e_1^1),f_n(e_1^2),f_n(e_1^3)$ would then have a nondegenerate intersection (containing $[f_n(v_0),f_n(v_1)]$), which is a contradiction because these three edges are inequivalent. So $f_n(v_1)=f_n(v_0)$, and since $f_n(e_1^i)$ has nondegenerate intersection with $f_n(e_0^i)$, it follows that $T_n$ is carried by $\tau'$. 
\end{proof}

\subsubsection{Folds}\label{sec-fold}

Lemma~\ref{fold-open} follows from the following proposition together with Lemma~\ref{trees-to-equiv}. 

\begin{prop}\label{fold-unique}
Let $\tau$ be an admissible train-track. Let $\tau'$ be a train-track obtained from $\tau$ by folding an illegal turn. 
\\ Then $\calp(\tau')$ is open in $\calp(\tau)$.
\end{prop}

\begin{proof}
By Lemma~\ref{subspace}, we have $\calp(\tau')\subseteq\calp(\tau)$. If $\tau'$ is obtained from $\tau$ by applying a singular fold, then Lemma \ref{lem_singular} ensures that $\calp(\tau')=\calp(\tau)$, so the conclusion holds. Denoting by $f:S^\tau\to T$ the carrying map, we can thus assume that $f(v_1)\neq f(v_2)$ (using the notations of the previous sections). Let $p:=f(e_1)\cap f(e_2)$. The endpoints of $p$, $f(e_1)$, and $f(e_2)$ are branch points in $T$ (because $\tau$ is admissible), so the lengths of these segments are determined by a finite set of translation lengths in $T$.

If $\tau'$ is obtained from $\tau$ by a partial fold, then $|p|$ is strictly less than both $|f(e_1)|$ and $|f(e_2)|$. Using the fact that the carrying map $f$ varies continuously with the carried tree (Lemma~\ref{carrier-unique}), we see that this property remains true for all trees in a neighborhood of $T$, and Lemma~\ref{fold-p} implies that all trees in this neighborhood are carried by $\tau'$.

Assume now that $\tau'$ is obtained from $\tau$ by a full fold, with $e_1$ fully folded into $e_2$. Let $v_d$ be the endpoint of an edge corresponding to a direction $d$ based at a vertex $v'_1$ equivalent to $v_1$, in the special gate. Then $|f(e_1)|<|f(e_2)|$, and $d_T(f(v_d),f(v_2))<d_T(f(v_d),f(v_1'))+d_T(f(v_1'),f(v_2))$. These are open conditions in $\overline{cv}_N$. Therefore, Lemma~\ref{fold-f-1} gives the existence of an open neighborhood $U$ of $T$ in $\calp(\tau)$ such that all trees in $U$ are carried by the same full fold. 
\end{proof}

\subsection{Control on the diameter: getting finer and finer decompositions}\label{sec-diam}

The goal of the present section is to prove Lemma~\ref{fold-diam}. The key lemma is the following.

\begin{lemma}\label{converge}
Let $\xi\in\partial_\infty FZ_N$, and let $(\tau_i)_{i\in\mathbb{N}}$ be a folding sequence of train-tracks directed by $\xi$. Then $(S^{\tau_i})$ converges to $\xi$ for the topology on $FZ_N\cup\partial_\infty FZ_N$. 
\end{lemma}

\begin{proof}
Let $T$ be a mixing representative of $\xi$. We can find a sequence $(S_n)_{n\in\mathbb{N}}$ of simplicial metric $F_N$-trees with trivial edge stabilizers such that
\begin{itemize}
\item for all $n\in\mathbb{N}$, the simplicial tree $S^{\tau_n}$ is obtained from $S_n$ by forgetting the metric,
\item for all $n\in\mathbb{N}$, the unique carrying map $f_n:S_n\to T$ is isometric on edges.
\end{itemize} 
\noindent In addition, for all $i<j$, there are natural morphisms $f_{ij}:S_i\to S_j$, such that $f_{ik}=f_{jk}\circ f_{ij}$ for all $i<j<k$. The sequence $(S_n)_{n\in\mathbb{N}}$ converges to a tree $S_\infty$: indeed, for all $g\in F_N$, the sequence $(||g||_{S_n})_{n\in\mathbb{N}}$ is non-increasing, so it converges. In addition, the tree $S_\infty$ is not reduced to a point, because the legal structure on $S_0$ induced by the morphisms $f_{0j}$ will stabilize as $j$ goes to $+\infty$, and a legal element (which exists because $\tau_0$ is admissible, and every legal turn for the train-track $\tau_0$ is also legal for the train-track structure on $S_0$ induced by the morphisms $f_{0j}$) cannot become elliptic in the limit. By taking a limit of the maps $f_{0j}$, we get an $F_N$-equivariant map $f_{0\infty}:S_0\to S_\infty$, which is isometric when restricted to every edge of $S_0$.

We will show that $S_\infty$ is $\calz$-averse and $\approx$-equivalent to $T$, which implies the convergence of $S^{\tau_i}$ to $\xi$ (for the topology on $FZ_N\cup\partial_\infty FZ_N$) by the continuity statement for the boundary map $\partial\pi$ given in Theorem~\ref{bdy-fz}. 

We first claim that the tree $S_\infty$ is not simplicial. Indeed, assume towards a contradiction that it is, and let $S'_\infty$ be the simplicial tree obtained from $S_\infty$ by adding the $f_{0\infty}$-images of all vertices of $S_0$ to subdivide the edges of $S_\infty$. Subdivide the edges of $S_0$ so that every edge of $S_0$ is mapped to an edge of $S'_\infty$. Since $T$ has trivial arc stabilizers, the folding process from $S_0$ to $S'_\infty$ never identifies two edges in the same orbit, so there is a lower bound on the difference between the volume of $S_{i+1}/F_N$ and the volume of $S_i/F_N$. Therefore, the folding process has to stop, contradicting the fact that the folding sequence is infinite. This shows that $S_\infty$ is not simplicial.

We now assume towards a contradiction that $S_\infty$ is not equivalent to $T$. There exists a $1$-Lipschitz $F_N$-equivariant map $f$ from $S_\infty$ to the metric completion of $T$, obtained by taking a limit of the maps $f_{n\infty}$ as $n$ goes to infinity, see the construction from \cite[Theorem 4.3]{Hor0}. Therefore, if $S_\infty$ has dense orbits, then there is a $1$-Lipschitz $F_N$-equivariant alignment-preserving map from $S_\infty$ to $T$ by \cite[Proposition 5.7]{Hor0}, which implies that $S_\infty$ is $\calz$-averse and equivalent to $T$ by \cite[Theorem~4.1]{Hor}. Therefore, its remains to show that $S_\infty$ has dense orbits. Assume towards a contradiction that it does not. Since $S_\infty$ is a limit of free and simplicial $F_N$-trees that all admit $1$-Lipschitz maps onto it, it has trivial arc stabilizers. Since $S_\infty$ is not simplicial, there is a free factor $A$ acting with dense orbits on its minimal subtree $S_A$ in $S_\infty$. By \cite[Lemma~3.10]{Rey}, the translates in $\overline{T}$ of $f(S_A)$ (which cannot be reduced to a point) form a transverse family in $\overline{T}$: this is a contradiction because the stabilizer of a subtree in a transverse family in a mixing $\calz$-averse tree cannot be a free factor (in fact, it cannot be elliptic in any $\calz$-splitting of $F_N$ by \cite[Proposition~4.23]{Hor}). 
\end{proof}

\begin{proof}[Proof of Lemma \ref{fold-diam}]
Let $\xi'\in P(\tau_i)$, and let $T'$ be a mixing representative of $\xi'$.  By Lemma~\ref{strong}, there exists a folding sequence of train-tracks $(\tau'_j)_{j\in\mathbb{N}}$ directed by $\xi'$, such that $\tau'_j=\tau_j$ for all $j<i$, and $\tau'_i$ is obtained from $\tau_i$ by a finite (possibly trivial) sequence of specializations. All trees $S^{\tau'_i}$ then lie on the image in $FZ_N$ of an optimal liberal folding path, which is a quasi-geodesic by \cite{Man}, and $(S^{\tau'_i})_{i\in\mathbb{N}}$ converges to $\xi'$ by Lemma~\ref{converge}. This shows that any quasi-geodesic ray in $FZ_N$ from $S^{\tau_0}$ to a point in $P(\tau_i)$ passes within bounded distance of $S^{\tau_i}$. As $S^{\tau_i}$ converges to $\xi\in\partial_\infty FZ_N$ (Lemma~\ref{converge}), it follows from the definition of a visual distance that the diameter of $P(\tau_i)$ converges to $0$.
\end{proof}

\section{End of the proof of the main theorem}\label{sec-proof}

We now sum up the arguments from the previous sections to complete the proof of our main theorem.

\begin{proof}[Proof of the main theorem]
We start with the case of $FZ_N$. The Gromov boundary $\partial_\infty FZ_N$ is a separable metric space (equipped with a visual metric). It can be written as the union of all strata $X_i$ (made of boundary points of index $i$), where $i\in\mathbb{N}^2$ can only take finitely many values. Each stratum $X_i$ is the union of the sets $P(\tau)$, where $\tau$ varies over the collection of all train-tracks of index $i$, and the collection of all train-tracks $\tau$ for which $P(\tau)$ is nonempty is countable by Lemma~\ref{countable}. By Proposition~\ref{closed}, each cell $P(\tau)$ is closed in its stratum $X_i$, and by Proposition~\ref{dim-0} it has dimension at most $0$. So each $X_i$ is a countable union of closed $0$-dimensional subsets, so $X_i$ is $0$-dimensional by the countable union theorem \cite[Lemma~1.5.2]{Eng}. Since $\partial_\infty FZ_N$ is a union of finitely many subsets of dimension $0$, the union theorem \cite[Proposition~1.5.3]{Eng} implies that $\partial_\infty FZ_N$ has finite topological dimension (bounded by the number of strata, minus $1$). In particular, the topological dimension of $\partial_\infty FZ_N$ is equal to its cohomological dimension, see the discussion in \cite[pp. 94-95]{Eng}. This gives the desired bound because the cohomological dimension of $\partial_\infty FZ_N$ is bounded by $3N-5$ by \cite[Corollary 7.3]{Hor} (this is proved by using the existence of a cell-like map from a subset of $\partial CV_N$ -- whose topological dimension is equal to $3N-5$ by \cite{GL95}, and applying \cite{RS}). 

It was shown in Section~\ref{sec-st} that the Gromov boundary $\partial_\infty FF_N$ is equal to the union of the strata $X_i$, with $i$ comprised between $(0,0)$ and $(0,2N-2)$. Therefore, the above argument directly shows that the topological dimension of $\partial_\infty FF_N$ is at most $2N-2$ (without appealing to cohomological dimension).

Finally, the Gromov boundary $\partial_\infty I_N$ is equal to a union of strata $X'_i$, with $i\in\mathbb{Z}_+$ comprised between $0$ and $2N-3$, where each stratum $X'_i$ is a subspace of the stratum $X_{(0,i)}$ from above. Given a train-track $\tau$ of index $(0,i)$, we let $P'(\tau):=P(\tau)\cap X'_i$. Being a subspace of $P(\tau)$, the set $P'(\tau)$ has dimension at most $0$. In addition $P'(\tau)$ is closed in its stratum $X'_i$, because its boundary in $\partial_\infty I_N$ is made of points of higher index by Proposition~\ref{closed}. The same argument as above thus shows that the topological dimension of $\partial_\infty I_N$ is at most $2N-3$.
\end{proof}

\appendix
\section{Why equivalence classes of vertices and specializations?}

\begin{figure}
\begin{center}
\def\JPicScale{0.65}
\input{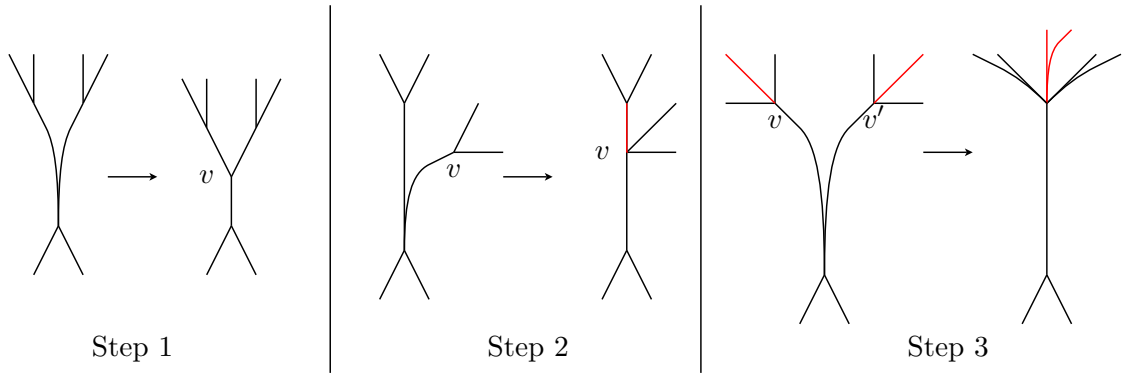}
\caption{The following three steps may repeat infinitely often along a folding sequence, showing that the homeomorphism type of the underlying simplicial tree might never stabilize.}
\label{fig-appendix}
\end{center}
\end{figure}

In this appendix, we would like to illustrate the reason why we needed to introduce an equivalence relation on the vertex set of $S^\tau$ in the definition of a train-track (Definition~\ref{d:train-track}), and allow for specializations in the definition of carrying (Definition~\ref{de-carry}).

\paragraph*{Why equivalence classes of vertices in the definition of a train-track?} In the world of train-tracks on surfaces, when one defines a train-track splitting sequence towards an arational foliation, the homeomorphism type of the complement of the train-track eventually stabilizes, after which each singularity in the limiting foliation is determined by a complementary region of a track, and the visible prongs determine the index of the singular point.

On the contrary, in a folding sequence of train-tracks towards an arational tree $T$ as defined in the present paper, the number of vertices in the preimage of a branch point in $T$, as well as the number of directions at these vertices, may never stabilize, as illustrated by the following situation (see Figure~\ref{fig-appendix}).

At some point on the folding sequence, one may have to perform a singular fold as depicted on the right-hand side of Figure~\ref{fig-appendix} (Step 3). If we had not declared the two identified vertices equivalent before computing the index, such an operation would have resulted in a drop in index. If such a situation could only occur finitely many times along the folding sequence, we could have defined a `stable' index along the folding sequence, but this might not be true in general. Indeed, later on along the folding sequence, a new trivalent vertex $v$ can be created due to a partial fold (Step 1 on Figure~\ref{fig-appendix}), and this exceptional vertex $v$ may then be declared equivalent to another vertex $v'$ in the track by applying a specialization: this results in the possibility of performing a new singular fold later on, identifying $v$ and $v'$. One could then hope that eventually, all singular folds in the folding sequence involve an exceptional vertex, which would not cause trouble as far as index is concerned; but even this may fail to be true in general. Indeed, it might happen that later on in the process, a full fold involving $v$ creates a fourth direction at $v$ (this is the red direction in Step 2 of Figure~\ref{fig-appendix}), which is not equivalent to any direction at $v$ (but is instead equivalent to a direction at another vertex $v'$ in the same class as $v$).

Introducing an equivalence class on the set of vertices in the definition of a train-track ensures that the index of our train-tracks remains constant along a folding path, and prevents overcounting the number of branch points in the limiting tree when counting branch points in the train-track. 

\paragraph*{Why introduce specializations in the definition of carrying?} We would finally like to give a word of motivation for the necessity to introduce specializations in the definition of carrying, instead of just saying that a train-track $\tau$ carries a tree $T$ if $\tau=\tau_f$ for some $F_N$-equivariant map $f:S^\tau\to T$. As already observed, in a folding sequence of train-tracks directed by $T$, partial folds create new trivalent vertices, and it can happen that the newly created vertex $v$ gets mapped to the same point in $T$ as another vertex $v'\in S^\tau$. We could have tried in this case to perform the partial fold and the specialization at the same time, in other words declare that there are several distinct partial folds of $\tau$, including one (denoted by $\tau_1$) where $v$ is not declared equivalent with any other vertex, and a second (denoted by $\tau_2$) where $v$ is declared equivalent to $v'$ (and there would be infinitely many other partial folds where $v$ is identified with any possible vertex from $S^\tau$). With this approach, the difficulty comes when proving openness of the set $\mathcal{P}(\tau_1)$ within $\mathcal{P}(\tau)$. Indeed, for a certain point $T \in \mathcal{P}(\tau)$, the carrying map may not identify the new trivalent vertex $v$ in $S^{\tau'}$ with any other vertex in $S^{\tau'}$ (hence $T$ is carried by $\tau_1$), while for nearby trees $T_n$ in $\mathcal{P}(\tau)$ (for the Gromov--Hausdorff topology), the carrying map identifies $v$ with a vertex $v_n$ going further and further away in $S^{\tau}$ as $n$ goes to infinity (hence no tree $T_n$ is carried by $\tau_1$). So without the extra flexibility in the definition of carrying, this would lead to $\mathcal{P}(\tau_1)$ not being open in $\mathcal{P}(\tau)$. This justifies our definition of carrying. We have a single way of performing a partial fold of a turn, and the specialization is performed later on along the folding sequence if needed.

\bibliographystyle{amsplain}
\bibliography{topdim-bib}

\small

\sc \noindent Mladen Bestvina, Department of Mathematics, University of Utah, 155 South 1400 East, JWB 233, Salt Lake City, Utah 84112-0090, United States

\noindent \tt e-mail:bestvina@math.utah.edu
\\
\\
\sc \noindent Camille Horbez, Laboratoire de Math\'ematiques d'Orsay, Univ. Paris-Sud, CNRS, Universit\'e Paris-Saclay, 91405 Orsay, France.

\noindent \tt e-mail:camille.horbez@math.u-psud.fr
\\
\\
\sc \noindent Richard D. Wade, Mathematical Institute, University of Oxford, Woodstock Road, Oxford, OX2 6GG, United Kingdom.

\noindent \tt email:wade@maths.ox.ac.uk

\end{document}